\documentclass[11pt]{article}
\usepackage[left=1in, right=1in, top=1in, bottom=1in, margin=1in]{geometry}

\usepackage{algorithm}
\usepackage{tcolorbox}
\usepackage{amsmath,amssymb,xspace,graphicx,relsize,bm,xcolor,breqn,algpseudocode,multirow}

\usepackage{thmtools} 
\usepackage{thm-restate}
\usepackage{lipsum}

\usepackage{tikz}
\usetikzlibrary{quantikz2}
\usepackage{adjustbox}

\usepackage{tcolorbox}
\newcommand{\Exp}{\mathop{\mathbb{E}}}

\usepackage[margin=1in]{geometry}

\usepackage{parskip}
\usepackage{amsthm}
\usepackage{dsfont}
\usepackage{array}
\usepackage{makecell}

\newcommand{\C}{\ensuremath{\mathcal{C}}}

\newcommand{\poly}{\ensuremath{\mathsf{poly}}}
\newcommand{\R}{\ensuremath{\mathbb{R}}}

\newcommand{\vspan}{\ensuremath{\mathsf{span}}}

\usepackage{enumerate}

\usepackage[pagebackref]{hyperref}
\usepackage[usestackEOL]{stackengine}
\usepackage{thm-restate,mathrsfs}
\usepackage{enumerate}
\usepackage{array}
\usepackage{parskip}
\def\01{\{0,1\}}
\hypersetup{
	colorlinks,
	linkcolor={blue!100!black},
	citecolor={teal!100!black},
}

\newcommand{\be}{\begin{equation}}
\newcommand{\ee}{\end{equation}}
\newcommand{\ba}{\begin{array}}
\newcommand{\ea}{\end{array}}
\newcommand{\bea}{\begin{eqnarray}}
\newcommand{\eea}{\end{eqnarray}}

\usepackage{mathtools}

\newcommand{\ra}{\rangle}
\newcommand{\la}{\langle}

\newcommand{\PFR}{\textsf{PFR}}

\newcommand{\norm}[1]{\left\lVert#1\right\rVert}

\newcommand{\F}{{\mathbb{F} }}

\newcommand{\FF}{\mathbb{F}}

\newcommand{\one}{\mathbf{1}}

\newtheorem{theorem}{Theorem}[section]
\newtheorem{definition}[theorem]{Definition}

\newtheorem{lemma}[theorem]{Lemma}

\newtheorem{claim}[theorem]{Claim}

\usepackage{thm-restate,mathrsfs}

\makeatletter
\def\widebreve{\mathpalette\wide@breve}
\def\wide@breve#1#2{\sbox\z@{$#1#2$}%
     \mathop{\vbox{\m@th\ialign{##\crcr
\kern0.08em\brevefill#1{0.8\wd\z@}\crcr\noalign{\nointerlineskip}%
                    $\hss#1#2\hss$\crcr}}}\nolimits}
\def\brevefill#1#2{$\m@th\sbox\tw@{$#1($}%
  \hss\resizebox{#2}{\wd\tw@}{\rotatebox[origin=c]{90}{\upshape(}}\hss$}
\makeatletter
\title{Algorithmic Polynomial Freiman-Ruzsa Theorems}
\author{
Srinivasan Arunachalam\thanks{IBM Quantum.}
\and
Davi Castro-Silva\thanks{University of Cambridge.\\DCS is supported by ESPRC Robust and Reliable Quantum Computing Grant.\\  TG is supported by ERC Starting Grant 101163189 and UKRI Future Leaders Fellowship MR/X023583/1.}
\and
Arkopal Dutt$^*$
\and
Tom Gur$^\dagger$ 
}
\begin{document}

\maketitle

\begin{abstract}
  We prove algorithmic versions of the polynomial Freiman-Ruzsa theorem of Gowers, Green, Manners, and Tao (Annals of Mathematics, 2025) in additive combinatorics.
    In particular, we give classical and quantum polynomial-time algorithms that, for $A \subseteq \mathbb{F}_2^n$ with doubling constant~$K$, learn an explicit description of a subspace $V \subseteq \mathbb{F}_2^n$ of size $|V| \leq |A|$ such that $A$ can be covered by $K^C$     translates of $V$, for a universal constant $C>1$.
\end{abstract}

\section{Introduction}
The Freiman-Ruzsa theorem \cite{freiman1987structure,ruzsa1999analog} is a cornerstone of additive combinatorics with diverse applications to theoretical computer science (cf. \cite{lovett2015exposition}). Loosely speaking, the theorem shows that sets exhibiting approximate combinatorial subgroup behaviour must be algebraically structured.  To make this precise, recall that  a set $A$ has doubling constant $K$ if $|A+A|\le K|A|$, where  $A+A = \{a+a' \;;\; a,a' \in A\}$. Note that $A$ has doubling constant $1$ if and only if it is a subgroup or a coset of a subgroup, and in turn, the doubling constant of $A$ can be thought of as a combinatorial measure of the approximate subgroup behaviour of sets. In this paper, we focus on subsets of $\mathbb{F}_2^n$. In this setting, the Freiman-Ruzsa theorem states that sets $A \subseteq \F_2^n$ with $|A+A|\le K|A|$ is covered by $\exp(K)$ translates of a subspace $V \subset \F_2^n$ of size $|V| \leq |A|$.

Marton conjectured that the aforementioned dependency in $K$ can be improved to a polynomial, in what became widely known as the Polynomial Freiman-Ruzsa (PFR) conjecture. Over a decade later, Sanders proved a quasipolynomial Bogolyubov-Ruzsa theorem, which implies a version of the Freiman-Ruzsa theorem with quasipolynomial dependency on the doubling constant $K$. In a recent breakthrough, the $\PFR$ conjecture was proved by Gowers, Green, Manners, and Tao.

\begin{theorem}[Combinatorial $\PFR$ theorem~\cite{gowers2023conjecture}]
\label{thm:marton_conjecture}
    There exists a polynomial $P_0: \R_+ \to \R_+$ such that the following holds.
    For any $n\geq 1$, if $A \subseteq \F_2^{n}$ satisfies $|A + A| \leq K |A|$, then $A$ is covered by at most $P_0(K)$ translates of a subspace $V \subset \F_2^n$ of size $|V| \leq |A|$.
\end{theorem}

A key reason for the importance of the $\PFR$ theorem in additive combinatorics is that it provides means to transition from a combinatorial notion of approximate subgroup structure, captured by constant doubling, to an algebraic notion, captured by a bounded subspace-cover, at only a polynomial cost. 

\subsection{Algorithmic $\PFR$}
The $\PFR$ theorem (and the closely-related quasi-polynomial Bogolyubov-Ruzsa theorem) also provide powerful tools that found diverse applications to theoretical computer science, including linearity testing of maps $f \colon \F_2^n \to \F_2^m$ \cite{samorodnitsky2007low}, constructions of two-source extractors from affine extractors \cite{zewi2011affine}, communication complexity lower bounds \cite{ben2014additive},
%in terms of the rank of their associated matrix
super-polynomial lower bounds on locally decodable codes \cite{bhowmick2013new}, constructions of non-malleable codes \cite{aggarwal2014non}, higher-order Goldreich-Levin theorems \cite{tulsiani2014quadratic,kim2023cubic,briet2025near}, sparsification algorithms for 1-in-3-SAT \cite{bedert2025strong}, quantum proofs for classical theorems \cite{drucker2009quantum}, quantum and classical worst-case to average-case reductions \cite{asadi2022worst,asadi2024quantum}, tolerant quantum testing of stabilizer states~\cite{ad2024tolerant}, quantum learning of structured stabilizer decompositions \cite{SCforstabilizers}, and beyond.

However, when considering applications of the $\PFR$ theorem and similar tools to theoretical computer science as above, it is often necessary or desirable to have an efficient algorithmic statement, where an explicit description of the subspace can be learned efficiently, as opposed to an existential combinatorial statement. Indeed, the naive algorithm that extracts the subspace runs in time $O(2^n)$. Fortunately, for the quasi-polynomial Bogolyubov-Ruzsa theorem, Ben-Sasson, Ron-Zewi, Tulsiani, and Wolf \cite{ben2014sampling} showed an algorithmic version, which extracts a subspace in time $O(n^3 \log n)$.

The above motivates a natural question that arose after the resolution of the $\PFR$ conjecture~\cite{gowers2023conjecture}. Namely, now that we know that a subset $A$ of constant doubling can be covered by at most a polynomial number of translates of a subspace $H \subset \F_2^n$ of size $|H| \leq |A|$, \emph{can we learn the subspace $H$ efficiently?}. We refer to this as the ``algorithmic $\PFR$ question." Our main contribution answers this question by proving an algorithmic version of the polynomial Freiman-Ruzsa theorem, where the covering subspace can be learned explicitly in $\poly(n)$-time. In the following, a query to a set $A \subseteq \F_2^n$ is an evaluation of the characteristic function $\one_A(x)$ for a chosen $x \in \F_2^n$, and a random sample from $A$ is a uniformly chosen element $a\in A$.

\begin{restatable}{theorem}{classicalPFR}
\label{thm:classicalPFR} (Algorithmic $\PFR$)
    There exists a polynomial $P_1: \R_+ \to \R_+$ such that the following holds.
    Let $A \subseteq \mathbb{F}_2^n$ satisfy $|A + A| \leq K|A|$ for a doubling constant $K \geq 1$.
    There is an $\tilde{O}(n^4)$-time randomized algorithm that uses $O(\log|A|)$ random samples and $\tilde{O}(\log^2 |A|)$ queries to $A$ which, with probability at least $2/3$, returns a subspace $V \leq \mathbb{F}_2^n$ of size at most $|A|$ such that $A$ can be covered by $P_1(K)$ translates of $V$.
\end{restatable}

We remark that the probability of success $2/3$ is arbitrarily chosen and can be amplified via standard error-reduction techniques. The algorithm returns the subspace $V$ specified by an explicit basis. As typically viewed in additive combinatorics, the doubling constant $K$ is a constant independent of $n$ (as constant doubling implies structure), and in turn our asymptotic notation suppresses factors of $K$.

\paragraph{Quantum algorithms.}
En route to obtaining our algorithmic $\PFR$ theorem, we obtain quantum algorithms with  query complexity (which we then dequantize; see the techniques section) and time complexity\footnote{In the context of quantum algorithms, by time we mean the total number of single and two-qubit quantum gates used in the quantum algorithm.} that is a factor-$n$ lesser than the classical algorithms.
Namely, we show the following~theorem.
\begin{restatable}{theorem}{quantumPFR}
\label{thm:quantumPFR} (Quantum Algorithmic $\PFR$)
    There exists a polynomial $P_1: \R_+ \to \R_+$ such that the following holds.
    Let $A \subseteq \mathbb{F}_2^n$ satisfy $|A + A| \leq K|A|$ for a doubling constant $K \geq 1$.
    There is an $O(n^3)$-time quantum algorithm that uses $O(\log|A|)$ random samples and quantum queries to $A$ which, with probability at least $2/3$, returns a subspace $V \leq \mathbb{F}_2^n$ of size at most $|A|$ such that $A$ can be covered by $P_1(K)$ translates of $V$.
\end{restatable}

\paragraph{Optimality.} We complement our algorithms by also showing that both of our classical and quantum algorithms are asymptotically optimal in terms of the \emph{query complexity} dependence in $n$, up to a logarithmic factor. 
In particular, we show that $\Omega(n^2)$ queries to the set $A$ are  necessary for classical algorithms in order to output the subspace $V$. Similarly, we show that $\Omega(n)$ quantum queries to the set $A$ are necessary. Our lower bounds are proven via a simple information-theoretic argument.  Finally,  we note that random samples are necessary to hit $A$ in the case it is sparse, and information-theoretically, it is necessary to obtain $\Omega(\log|A|)$ samples from $A$ in order to hit at least a basis for~$A$. See Section~\ref{sec:lowerbounds} for the lower bounds.

\subsection{Homomorphism testing and structure-vs-randomness decomposition}
\label{sec:app}
A key reason for the power and centrality of the $\PFR$ theorem is its applications to deriving strong structural theorems regarding homomorphism testing and structure-vs-randomness decomposition. However, while the $\PFR$ theorem is known to be equivalent to the aforementioned structural theorems \cite{green2004finite,green2005notes}, the equivalences are not trivially algorithmic. Nonetheless, we provide versions of these theorems that admit efficient algorithms.

The first theorem is concerned with local-to-global phenomena. Namely, it shows that if a map $f: \F_2^m \to \F_2^n$ satisfies a local affine-linear constraint with a significant probability, then it must be globally close, in fractional distance, to an affine-linear map, which can be efficiently learned. This statement is useful because it connects the $\PFR$ theorem to property testing and coding theory.

\begin{theorem}[Homomorphism testing]
\label{thm:PFR2}
    There exists a polynomial $P_2: \R_+ \to \R_+$ such that the following holds.
    Let $f: \F_2^m \to \F_2^n$ satisfy
    $$\Pr_{x_1+x_2 = x_3+x_4} \big[f(x_1)+f(x_2) = f(x_3)+f(x_4)\big] \geq 1/K.$$
    Then, there is an affine-linear function $g: \F_2^m \to \F_2^n$ such that $f(x) = g(x)$ for at least $2^m/P_2(K)$ values of $x\in \F_2^m$. Furthermore, there is an $\tilde{O}((m+n)^3)$-time randomized algorithm that, with probability at least $2/3$, 
    learns a concise representation of $g$.
\end{theorem}

The second theorem also exhibits local-to-global structure, this time by showing that maps that are locally an approximate homomorphisms admit a structured decomposition, which can be efficiently learned.

\begin{theorem}[Structured approximate homomorphism]
\label{thm:PFR3}
    There exists a polynomial $P_3: \R_+ \to \R_+$ such that the following holds.
    Let $f: \F_2^m \to \F_2^n$ satisfy
    $$\big|\big\{f(x)+f(y)-f(x+y):\: x, y\in \F_2^m \big\}\big| \leq K.$$
    Then $f$ may be written as $g+h$, where $g: \F_2^m \to \F_2^n$ is linear and $|\mathrm{Im}(h)| \leq P_3(K)$. Furthermore, there is an $\tilde{O}((m+n)^3)$-time randomized algorithm that, with probability at least $2/3$, learns a concise representation of $g$.
\end{theorem}

\subsection{Technical overview}
A natural approach for proving an algorithmic $\PFR$ theorem is to try to algorithmize each step in the proof of the $\PFR$ conjecture given in \cite{gowers2023conjecture}, which would in principle provide a result similar to Theorem~\ref{thm:classicalPFR}.
Unfortunately, the aforementioned proof heavily relies on entropic methods that are non-algorithmic by nature and it is unclear whether such machinery can be transformed into efficient algorithms. To overcome this barrier, we take a detour through quantum algorithms.

\paragraph{Stabilizer learning and the Gowers $U^3$-norm of quantum states.} 
In a recent breakthrough, Chen, Gong, Ye, and Zhang provided an efficient quantum procedure to learn the closest stabilizer state to a given quantum state \cite{chen2024stabilizer}.
The connection between this task and additive combinatorics was first noted by Arunachalam and Dutt \cite{ad2024tolerant}, who defined the Gowers $U^3$-norm of an arbitrary $n$-qubit quantum state $\ket{\psi} = \sum_{x \in \FF_2^n} f(x) \ket{x}$ (where $(f(x))_x$ is an $\ell_2$-unit vector) as
\begin{equation}
    \norm{\ket{\psi}}_{U^3} = 2^{n/2} \left[ \Exp_{x,h_1,h_2,h_3 \in \mathbb{F}_2^n} \prod_{\omega \in \mathbb{F}_2^3} C^{|\omega|} f(x + \omega \cdot h) \right]^{1/2^{3}},
\end{equation}
where $C^{|\omega|}f = \overline{f}$ if $\omega:=\sum_{j \in [3]} \omega_j$ is odd and $f$ otherwise. This is proportional to the Gowers $U^3$-~norm of the function $f$ encoded in the amplitudes of the state $\ket{\psi}$, but normalized so that $\norm{\ket{\psi}}_{U^3} \leq \norm{f}_{\ell_2}~=~1$.

Arunachalam and Dutt showed that stabilizer states are the extremizers of the Gowers $U^3$-norm over quantum states, and that a quantum state has non-negligible $U^3$-norm if and only if it correlates with a stabilizer state. Moreover, using the $\PFR$ theorem (Theorem~\ref{thm:marton_conjecture}), they obtained a \emph{polynomial Gowers inverse theorem} for quantum states:
the $U^3$-norm of a quantum state and its maximal correlation with a stabilizer state are polynomially related (see also \cite{bao2025tolerant, mehraban2024improved}). As such, our high-level strategy is to use the stabilizer learning protocol in \cite{chen2024stabilizer} to obtain a quantum algorithmic version of the polynomial Gowers inverse theorem, which is known to also be equivalent to the $\PFR$ theorem due to work of Green and Tao \cite{green2010equivalence} and of Lovett \cite{lovett2012equivalence}.
We then arrive at an algorithmic result, albeit quantum, of a statement that is combinatorially equivalent to $\PFR$.

However, as discussed in the previous section, it is non-trivial to make such combinatorial equivalences algorithmic. In this paper, we algorithmize a proof of equivalence between these two results (inspired by the proofs of Green-Tao and Lovett), thus allowing us to employ the stabilizer learning algorithm \cite{chen2024stabilizer} to obtain efficient \emph{quantum} algorithms for the $\PFR$ theorem and the structural theorems stated in Section~\ref{sec:app}.

\paragraph{Classical algorithms via dequantization.} After obtaining the quantum algorithms above, the last ingredient needed is a method to dequantize these algorithm so as to obtain efficient \emph{classical} algorithms for $\PFR$, as stated in Theorems~\ref{thm:classicalPFR},~\ref{thm:PFR2} and~\ref{thm:PFR3}.
Towards this end, we use in our arguments the machinery developed by Bri\"{e}t and Castro-Silva \cite{briet2025near} to replace the quantum learning algorithm in \cite{chen2024stabilizer} by a classical algorithm that emulates it.
This allows us to obtain analogous algorithmic results in the classical setting, at the expense of quadratically worse query complexity than in the quantum setting. This quantum-to-classical blow-up is inherent, and indeed, we prove it is necessary and essentially optimal (see Theorem~\ref{thm:classicallowerbound}).

\paragraph{Proof outline.}
With the strategy above in mind, we proceed to give a high-level outline of the proofs of Theorems~\ref{thm:classicalPFR} and~\ref{thm:quantumPFR}, which build on the combinatorial arguments of Green and Tao \cite{green2010equivalence} and Lovett \cite{lovett2012equivalence}.
Our other algorithmic results follow via similar methods from these two theorems and the combinatorial arguments of Green and Ruzsa \cite{green2005notes}.

Suppose we have sample and query access to a set $A\subseteq \F_2^n$ such that $|A + A| \leq K|A|$ for a doubling constant $K \geq 1$.
%\paragraph{Localization.}
We shall first need to localize $A$ inside the space $\F_2^n$, which in applications can be much larger than $A$ itself (indeed, this is the reason why sample access to $A$ is necessary). We do this by first sampling $O(\log |A|)$ uniformly random elements from $A$ and taking their linear span, which we denote by $U \leq \F_2^n$. While $U$ might not contain all of the original set, using the fact that $A$ has bounded doubling we can show that their intersection $A' := A\cap U$ will likely comprise at least half of the points in $A$
(for a careful choice of parameters).
We prove this in Lemma~\ref{lem:spansample}, which allows us to shift attention from $A$, which can be arbitrarily sparse inside $\F_2^n$, to the localization $A'$, which occupies a positive fraction (at least a $2^{-2K}$-fraction) of the vector space $U$, by the Freiman-Ruzsa theorem.
Note that $A'$ will also have small doubling constant:
$$|A'+A'| \leq |A+A| \leq K|A| \leq 2K|A'|.$$

We next wish to obtain a ``dense model'' of the localized set $A'$; that is, a set $S \subseteq \F_2^m$ that is ``additively equivalent'' to $A'$, as captured by the notion of Freiman isomorphisms, but which has density at least $1/K^C$ (for a universal constant $C>1$) inside its ambient space $\F_2^m$.
We show how to do this in Lemma~\ref{lem:randomFreiman}, which states that with high probability a uniformly random linear map $\pi: U \to \F_2^m$, for $m = \log |4A'| + 10$, will be a Freiman isomorphism (i.e., isomorphism of additive quadruples) from $A'$ to $S := \pi(A')$.
Informally, this means they have the same additive structure, and hence such a dense model can be efficiently obtained by sampling.

Equipped with the localized dense model, we proceed to learn the covering subspace. Denote by $f: S \to A'$ the inverse of $\pi$ when restricted to $S$.
By the definition of Freiman isomorphisms, we have that
$$\forall a, b, c, d \in S:\: a+b = c+d \implies f(a)+f(b) = f(c)+f(d).$$
From this approximate linearity condition of $f$ on $S$, we can show that the function
$$g(x, y) = \one_S(x) (-1)^{f(x)\cdot y}$$
is approximately quadratic.

In particular, following the approach of Green~\cite{green2005notes} and Green-Tao~\cite{green2010equivalence}, we show a slight strengthening of the homomorphism testing formulation of the $\PFR$ theorem (see Lemma~\ref{lem:closeaffine}), which we then proceed to algorithmize (see Lemma~\ref{lem:findingaffine}), relying on tools such as the quadratic Goldreich-Levin theorem (see Theorem~\ref{thm:quadratic_GL}). In more detail, we first prove that there exists a quadratic function $q: \F_2^{m+n} \to \F_2$ such~that
\begin{equation} \label{eq:quadcorr}
    \Big|\Exp_{x\in \F_2^m,\, y\in \F_2^n} \one_S(x) (-1)^{f(x)\cdot y} (-1)^{q(x,y)}\Big| \geq \frac{1}{P(K)} \;,
\end{equation}
for some polynomial $P\colon \R_+ \to \R_+$.

Crucially, we can efficiently \emph{learn} such a high-correlation quadratic function $q$.
This is done relying on the stabilizer learning algorithm of Chen et al \cite{chen2024stabilizer} in the quantum setting, or its dequantization by Bri\"{e}t and Castro-Silva \cite{briet2025near} in the classical setting.
In order to use those theorems, however, we need to be able to efficiently query the function $g(x,y)$.
This requires making queries to the set $S = \pi(A')$, and inverting the linear map $\pi$ restricted to $A'$.
We show how this can be done using a $O(n^3)$ time pre-processing step,
%\snote{should this be $n^4/n^3$? I guess somewhere we need to say the entire cost is $n^4$ to keep it consistent with the main theorem statement?}
and an extra cost of $O(n^2)$ time
%and $2^{2K}$ queries to $A$
for each query to $g$.
The total time and query complexities of our algorithms follow from this step.

From Eq.~\eqref{eq:quadcorr} and algebraic manipulations, we conclude that the homogeneous bilinear~form
$$B(x, y) = q(x,y) - q(x,0) - q(0,y) + q(0,0)$$
correlates well with $g(x,y)$.
Since we obtained an explicit description of $q$, we can compute a matrix $M\in \F_2^{n\times m}$ such that $B(x,y) = y^T Mx$.
By a simple Fourier analytic argument,
%simple manipulations involving the Fourier transform of $g$ and Parseval's identity,
we can then conclude there is some $v\in \F_2^n$ such that
$$f(x) = Mx + v \quad \text{for at least $2^m/P'(K)$ values $x\in S$,}$$
where $P'\colon \R_+ \to \R_+$ is another polynomial we obtain in our proof.
This implies that the subspace
$$V = \big\{Mx:\: x\in \F_2^m\big\} \leq \F_2^n$$
satisfies
$$\big|A'\cap (v+V)\big| = \big|\mathrm{Im}(f)\cap (v+V)\big| \geq 2^m/P'(K) \geq |A|/P'(K).$$
By an application of Ruzsa's covering lemma, we conclude that $P'(K)$ translates of $V$ can cover~$A$.
By our choice of $m$ and the fact that $A'$ has bounded doubling, we deduce that $V$ is covered by $2^m/|A| \leq 2^{12} K^4$ translates of any of its subspaces having size $|A|$, which concludes the high-level overview of the proof.

\paragraph{Future directions.} As discussed above, a crucial component of the proof of our \emph{classical} algorithmic $\PFR$ theorem relies on dequantizing a \emph{quantum} algorithm for learning the closest stabilizer state to a given arbitrary quantum state. Finding further connections between quantum algorithms and additive combinatorics is an interesting direction for future research. One concrete question that is left open from our work is to improve the dependence on the doubling constant $K$. As is standard in additive combinatorics, $K$ is assumed to be a constant, but for asymptotically growing $K$ it is an interesting open problem whether there exists an algorithm with query and time complexities that scale polynomially in $K$.

\section{Preliminaries}

In this section, we state results that we will use in our algorithmic $\PFR$ theorems. We begin with standard definition in additive combinatorics.

\subsection{Additive combinatorics}

For a set $S\subseteq \FF_2^n$ and $k\geq 1$, define the $k$-fold sumset as $kS=\{\sum_{i\in T}a_i:a_i\in S\}_{|T|=k}$. In particular, $S+S = \{a_1+a_2:a_1,a_2\in S\}$.
We define the \emph{doubling constant} of $S$ as the smallest integer $K$ such that $|2S|\leq K|S|$.
We denote the linear span of $S$ by $\vspan(S)$. An additive quadruple in a set $S$ is $(x_1, x_2, x_3, x_4)\in S^4$ such that $x_1+x_2 = x_3+x_4$.

\begin{definition}
   The \emph{additive energy} of a set $S$ is the number of additive quadruples in $S$:
$$E(S) := \big|\big\{(x_1, x_2, x_3, x_4)\in S^4:\: x_1+x_2 = x_3+x_4\big\}\big|.$$ 
Note that $E(S) \leq |S|^3$.
\end{definition}

\begin{definition}[Freiman homomorphism]
For a set $S\subseteq \F_2^n$, a function $\phi: S\rightarrow \F_2^m$ is a \emph{Freiman homomorphism} if, for every additive quadruple $x_1, x_2, x_3, x_4\in S$ such that $x_1+x_2=x_3+x_4$, we have that $\phi(x_1)+\phi(x_2)=\phi(x_3)+\phi(x_4)$.
\end{definition}

\begin{definition}[Freiman isomorphism]
   A \emph{Freiman isomorphism} is a bijective Freiman homomorphism $\phi$ such that its inverse is also a Freiman homomorphism;
this is equivalent to requiring~that
$$
\forall a, b, c, d\in S:\: a+b = c+d \iff \phi(a)+\phi(b) = \phi(c)+\phi(d).
$$ 
\end{definition}
We shall need the following well-known theorems and lemmas in additive combinatorics.

\begin{theorem}[Freiman-Ruzsa theorem \cite{ruzsa1999analog, even2012sums}]
\label{thm:spanAsize}
    Let $A\subseteq \FF_2^n$.
    If $|A+A|\leq K\cdot |A|$, then $|\vspan(A)| \leq  2^{2K}/(2K) \cdot|A|$.
\end{theorem}

\begin{theorem}[Balog-Szemer\'{e}di-Gowers theorem]
\label{thm:BSG}
    Let $A \subseteq \F_2^n$ be a set such that
    $$E(A) = \big|\big\{(x_1, x_2, x_3, x_4)\in A^4:\: x_1+x_2 = x_3+x_4\big\}\big| \geq |A|^3/K.$$
    There there is a set $A' \subseteq A$ such that
    $$|A'| \geq |A|/P_{BSG}^{(1)}(K) \quad \text{and} \quad |A'+A'| \leq P_{BSG}^{(2)}(K) |A|.$$
\end{theorem}

\begin{lemma}[Pl\"unnecke's inequality~\cite{tao2006additive}]
\label{lem:4Asize}
    If $A \subseteq \FF_2^{n}$ satisfies $|2A|/|A| \leq K$, then $|4A|/|A| \leq K^4$.
\end{lemma}

\begin{lemma}[Ruzsa's covering lemma]
\label{lem:ruzsacovering}
    If $S, T \subseteq \FF_2^n$ satisfy $|T+S| \leq K|S|$, then there is a subset $X \subseteq T$ of size $|X| \leq K$ such that $T \subseteq X + 2S$.
\end{lemma}

\subsection{Quantum information}

Our quantum algorithms revolve around stabilizer states; we will use the following characterization of stabilizer states.
\begin{theorem}[\cite{nest2008classical}]
\label{thm:neststabilizer}
    Every $k$ qubit stabilizer state can be written in the following form
    $$
    \frac{1}{\sqrt{|A|}}\sum_{x\in A}i^{\ell(x)}(-1)^{q(x)}\ket{x},
    $$
    for some affine subspace $A\subseteq \F_2^k$, quadratic polynomial $q$ and linear polynomial $\ell$ in the variables $(x_1,\ldots,x_k)\in \F_2^k$. 
\end{theorem}

Our quantum algorithmic $\PFR$ theorems will crucially use the agnostic learnability of stabilizer states. Informally the task here is as follows: supposing an arbitrary quantum state $\ket{\psi}$ was $\tau$-close to an \emph{unknown} stabilizer state $\ket{\phi}$, output the ``nearest" stabilizer state $\ket{\phi'}$ that is $\tau-\varepsilon$ close. A recent work of Chen, Gong, Ye, and Zhang~\cite{chen2024stabilizer}, gave an agnostic learning algorithm that runs in time quasipolynomial in $1/\tau$ and polynomial in the other parameters. Formally, their result is stated in the following theorem.

\begin{theorem}[\cite{chen2024stabilizer}]
\label{thm:sitanbootstrapping}
Let $\mathcal{C}$ be  the class of stabilizer states. Fix any $\varepsilon\leq \tau \in (0,1)$. 
There is an algorithm that, given access to copies of $\rho$ with $\max_{|\phi'\rangle \in \mathcal{C}} |\langle \phi' | \rho | \phi' \rangle| \geq \tau$, outputs a $|\phi\rangle \in \C$ such that $|\langle \phi | \rho | \phi \rangle| \geq \tau - \varepsilon$ with high probability.
The algorithm performs single-copy and two-copy measurements on at most $n\cdot \poly(1/\varepsilon,(1/\tau)^{\log 1/\tau})$ copies of $\rho$ and runs in time $n^3 \poly(1/\varepsilon, (1/\tau)^{\log 1/\tau})$.
\end{theorem}
We will also require the following subroutines for estimating the overlap between two states and obtaining unitaries that prepare stabilizer states
\begin{lemma}[\textsf{SWAP} test]
\label{lem:swap_test}
Let $\varepsilon,\delta \in (0,1)$. Given two arbitrary $n$-qubit quantum states $\ket{\psi}$ and~$\ket{\phi}$, there is a quantum algorithm that estimates $|\langle \psi | \phi \rangle|^2$ up to error $\varepsilon$ with probability at least $1-\delta$ using $O(1/\varepsilon^2\cdot \log(1/\delta))$ copies of $\ket{\psi},\ket{\phi}$ and runs in $O(n/\varepsilon^2\cdot \log(1/\delta))$ time.
\end{lemma}

\begin{lemma}[Clifford synthesis~\cite{dehaene2003clifford,patel2003efficient}]\label{lem:clifford_synthesis}
Given the classical description of an $n$-qubit stabilizer state $\ket{\phi}$, there is a quantum algorithm that outputs a Clifford circuit $C$ that prepares $\ket{\psi}$, using $O(n^2/\log n)$ many single-qubit and two-qubit Clifford gates.
\end{lemma}

\subsection{Classical learning algorithms}

Our classical algorithmic $\PFR$ theorems will crucially use the recent result by Bri{\"e}t and Castro-Silva~\cite{briet2025near}, which can be thought of as dequantizing the agnostic learning algorithm for stabilizer states of~\cite{chen2024stabilizer}.
In particular, they showed that if one had \emph{query access} to the amplitude vector of the unknown quantum state $\ket{\psi}$,
%(assumed to be a ``flat'')
then one can obtain a quadratic polynomial that most closely approximates the amplitudes.
Formally, they proved the following theorem.

\begin{theorem}[{\cite[Theorem~1.1]{briet2025near}}]
\label{thm:quadratic_GL}
Let $f: \FF_2^n \rightarrow [-1,1]$ be a $1$-bounded function and let $\varepsilon,\delta > 0$.
There is a randomized algorithm $\mathcal{A}$ that makes $n^2 \log n \log(1/\delta) (1/\varepsilon)^{O(\log(1/\varepsilon))}$ queries to $f$ and has $O(n^3)$ time complexity such that, with probability at least $1-\delta$, outputs a quadratic polynomial $p: \FF_2^n \rightarrow \FF_2$ satisfying
$$
|\Exp_{x \in \FF_2^n} [f(x) (-1)^{p(x)}]| \geq \max_{q \text{ quadratic }} |\Exp_{x \in \FF_2^n} [f(x) (-1)^{q(x)}]| - \varepsilon.
$$
\end{theorem}

\section{Algorithmic lemmas}
\label{sec:lemmas}

In this section, we will prove lemmas that will be useful in establishing our quantum and classical algorithmic $\PFR$ theorems.
\subsection{Probabilistic dense model and sparse set localization}

The following lemmas provide efficient randomized algorithms for finding Freiman isomorphisms and localizing sparse sets by showing that there is an abundance of them, and hence they can be sampled at random efficiently.

\begin{lemma}[Localizing a sparse set]
\label{lem:spansample}
   Let $\varepsilon, \delta >0$ and $A \subseteq \FF_2^n$. Let $m := \log |\vspan(A)|$, $k := \lceil 2m/\varepsilon\rceil \cdot \lceil\log(1/\delta)\rceil$.
    If $v_1, \dots, v_k$ are uniformly random elements of~$A$ then, with probability at least $1-\delta$, we have
    $$|A\cap \vspan\{v_1, \dots, v_k\}| \geq (1-\varepsilon)|A|.$$
\end{lemma}

\begin{proof}
Let $\ell \geq 2$ be an integer to be chosen later, and let $v_1, \dots, v_\ell$ be $\ell$ independent random elements of $A$.
Let $V_0 = \{0\}$ and, for each $1\leq i\leq \ell$, denote the linear span of the first $i$ random elements $v_1, \dots, v_i$ by $V_i$. Suppose first that
\begin{equation}
\label{eq:probbound}
    \Pr_{v_1, \dots, v_\ell\in A}\big[|A\cap V_\ell| \geq (1-\varepsilon) |A|\big] < 1/2.
\end{equation}
Then $\Pr_{v_1, \dots, v_\ell\in A}\big[|A\setminus V_\ell| > \varepsilon |A|\big] \geq 1/2$, and so
\begin{align}
    \label{eq:lowerboundonA-Vi}
    \Pr_{v_1, \dots, v_i\in A}\big[|A\setminus V_i| \geq \varepsilon |A|\big] > 1/2 \quad \text{for all $0\leq i \leq \ell$.}
\end{align}
It follows that
\begin{align*}
    \Exp_{v_1, \dots, v_\ell\in A} \big[\dim(V_\ell)\big] &= \sum_{i=1}^\ell \Pr_{v_1, \dots, v_i\in A}\big[v_i \notin V_{i-1}\big] \\
    &\geq \sum_{i=1}^\ell \Pr_{v_1, \dots, v_{i-1}\in A}\Big[|A\setminus V_{i-1}| \geq \varepsilon |A|\Big] \cdot \Pr_{v_1, \dots, v_i\in A}\Big[v_i \notin V_{i-1} \Big| |A\setminus V_{i-1}| \geq \varepsilon |A|\Big] \\
    &> \sum_{i=1}^\ell 1/2 \cdot \varepsilon
    = \varepsilon \ell/2,
\end{align*}
where the final inequality used Eq.~\eqref{eq:lowerboundonA-Vi}. 
Since $V_\ell \subseteq \vspan(A)$, we must have $\dim(V_\ell) \leq \log |\vspan(A)| = m$, and thus $\ell < 2m/\varepsilon$ is required for equation~\eqref{eq:probbound} to hold. Denoting $t := \lceil 2m/\varepsilon\rceil$, we conclude~that
\begin{equation*}
    \Pr_{v_1, \dots, v_t\in A}\big[|A\cap \vspan\{v_1, \dots, v_t\}| \geq (1-\varepsilon) |A|\big] \geq 1/2.
\end{equation*}
Repeating this sampling $\lceil\log(1/\delta)\rceil$ times independently at random, the probability that we succeed at least once is at least $1-\delta$.
With $k = t \lceil\log(1/\delta)\rceil$, it follows that
$$
\Pr_{v_1, \dots, v_k\in A}\Big[|A\cap \vspan\{v_1, \dots, v_k\}| \geq (1-\varepsilon) |A|\Big] \geq 1-\delta,
$$
proving the lemma statement.
\end{proof}

\begin{lemma}[Algorithmic dense model]
\label{lem:randomFreiman}
    Let $\delta > 0$,  $A\subseteq \FF_2^n$ and  let $m \geq \log |4A| + \log 1/\delta$ be an integer.
    Suppose $\pi: \FF_2^n \rightarrow \FF_2^m$ is a random linear map.
    Then $A$ is Freiman-isomorphic to $\pi(A)$ with probability at least $1-\delta$.
\end{lemma}

\begin{proof}
Recall that $\pi$ is a Freiman isomorphism between $A$ and $\pi(A)$ if
$$
\forall a, b, c, d\in A:\: a+b+c+d = 0 \iff \pi(a)+\pi(b) + \pi(c)+\pi(d) = 0.
$$
Observe that the property above implies that $\pi$ is bijective.\footnote{To see this, if $c=d$, then the implication above gives $a=b \iff \pi(a)=\pi(b)$, which implies $\pi$ is a bijection.}
If $\pi: \FF_2^n \rightarrow \FF_2^m$ is a linear map, then the forward implication is automatically satisfied, and moreover
$$\pi(a)+\pi(b) + \pi(c)+\pi(d) = \pi(a+b+c+d).$$
It then suffices to check that
$$\forall a, b, c, d\in A:\: \pi(a+b+c+d) = 0 \implies a+b+c+d = 0,$$
which is equivalent to requiring that $\pi(x) \neq 0$ for all nonzero $x\in 4A$.

Now let $\pi: \FF_2^n \rightarrow \FF_2^m$ be a uniformly random linear map.
Then, for each $x\in 4A\setminus\{0\}$ individually, $\pi(x)$ is uniformly distributed over $\FF_2^m$.
It follows from the union bound that
$$
\Pr\big[\exists x\in 4A\setminus\{0\}:\: \pi(x) = 0\big] \leq \frac{|4A|-1}{2^m},
$$
which is less than $\delta$ if $2^m \geq |4A|/\delta$.
This concludes the proof.
\end{proof}

\subsection{Algorithmic restricted homomorphism}
We will show the following slight strengthening of the homomorphism testing formulation of the $\PFR$ theorem, which we prove
following the approach of Green~\cite{green2005notes} and Green-Tao~\cite{green2010equivalence}, and then algorithmize.

\begin{lemma}[Restricted homomorphism testing]
\label{lem:closeaffine}
Suppose $S\subseteq \F_2^m$ and $f: S \to \F_2^n$ satisfy
$$\big|\big\{(x_1, x_2, x_3, x_4)\in S^4:\: x_1+x_2 = x_3+x_4 \,\text{ and }\, f(x_1)+f(x_2) = f(x_3)+f(x_4)\big\}\big| \geq 2^{3m}/K.$$
Then there exists an affine-linear function $\psi: \F_2^m \to \F_2^n$ such that $f(x) = \psi(x)$ for at least $2^m/P_4(K)$ values of $x\in S$.
\end{lemma}

\begin{proof}
Consider the ``graph'' set
$$\Gamma = \big\{(x, f(x)):\: x\in S\big\} \subseteq \F_2^{m+n}.$$
Then $|\Gamma| = |S| \leq 2^m$ and $E(\Gamma) \geq 2^{3m}/K \geq |\Gamma|^3/K$.
By the Balog-Szemer\'{e}di-Gowers Theorem (Theorem~\ref{thm:BSG}), there exists a set $\Gamma' \subseteq \Gamma$ such that
$$|\Gamma'| \geq |\Gamma|/P_{BSG}^{(1)}(K) \quad \text{and} \quad |\Gamma'+\Gamma'| \leq P_{BSG}^{(2)}(K)\cdot |\Gamma|.$$
By the combinatorial $\PFR$ theorem (Theorem~\ref{thm:marton_conjecture}), it then follows that $\Gamma'$ can be covered by
$$K' := P_1\big(P_{BSG}^{(1)}(K) P_{BSG}^{(2)}(K)\big)$$
translates of a subspace $H\leq \F_2^{m+n}$ of size $|H| \leq |\Gamma'|$, say
$$
\Gamma' \subseteq \bigcup_{i=1}^{K'} (u_i + H).
$$
Let $\pi: \F_2^{m+n} \to \F_2^m$ denote the projection map onto the first $m$ coordinates:
$\pi(x, y) = x$ for $x\in \F_2^m$, $y\in \F_2^n$.
Let $\ker_H(\pi) = H \cap \big(\{0^m\} \times \F_2^n\big)$ be the kernel of $\pi$ restricted to $H$ and let $H'$ be a complemented subspace of $\ker_H(\pi)$ in $H$, so that $H = H' \oplus \ker_H(\pi)$.
By linearity and the injectivity of $\pi$ on $H'$, there exists a matrix $M\in \F_2^{n\times m}$ such that
\begin{equation}
\label{eq:Hprime}
    H' = \big\{(x, Mx):\: x\in \pi(H)\big\},
\end{equation}
and by the rank-nullity theorem we have that
$$|H| = |\ker_H(\pi)|\cdot |\pi(H)|.$$
Moreover, since $\Gamma'$ is a graph, we have
$$\big|\Gamma' \cap (u_i+H)\big| = \big|\pi\big(\Gamma' \cap (u_i+H)\big)\big| \leq |\pi(u_i+H)| = |\pi(H)|,$$
and thus
$$|\Gamma'| = \Bigg|\Gamma'\cap \bigcup_{i=1}^{K'} (u_i+H)\Bigg| \leq \sum_{i=1}^{K'} \big|\Gamma'\cap (u_i+H)\big| \leq K' |\pi(H)|,$$
from which we conclude that $|\pi(H)| \geq |\Gamma'|/K'$.
Finally, since $H = H' \oplus \ker_H(\pi)$, we have
$$|\Gamma'| = \Bigg|\Gamma'\cap \bigcup_{i=1}^{K'} \bigcup_{v\in \ker_H(\pi)} (u_i+v+H')\Bigg| \leq \sum_{i=1}^{K'} \sum_{v\in \ker_H(\pi)} \big|\Gamma'\cap (u_i+v+H')\big|;$$
there must then exist some translate $u_i$ and some $v\in \ker_H(\pi)$ such that
$$\big|\Gamma'\cap (u_i+v+H')\big| \geq \frac{|\Gamma'|}{K' |\ker_H(\pi)|}.$$
Using the assumption $|H| \leq |\Gamma'|$, the identity $|H| = |\ker_H(\pi)|\cdot |\pi(H)|$ and the bound $|\pi(H)| \geq |\Gamma'|/K'$, we conclude from the last inequality that
$$\big|\Gamma'\cap (u_i+v+H')\big| \geq \frac{|H|}{K' |\ker_H(\pi)|} = \frac{|\pi(H)|}{K'} \geq \frac{|\Gamma'|}{(K')^2}.$$

We can now easily conclude.
Fixing $u_i = (x_1, y_1)$, $v = (x_2, y_2) \in \F_2^{m+n}$ such that the above inequality holds, we obtain from the description of $H'$ (equation~\eqref{eq:Hprime}) that
\begin{align*}
    \Gamma'\cap (u_i+v+H') &= \Gamma' \cap \big\{(x+x_1+x_2,\, Mx+y_1+y_2):\: x\in \pi(H)\big\} \\
    &= \Gamma' \cap \big\{(x,\, Mx - Mx_1 - Mx_2 + y_1 + y_2):\: x\in \pi(H)+x_1+x_2\big\}.
\end{align*}
There must then be at least $|\Gamma'|/(K')^2$ values of $x\in S$ such that $(x, f(x)) \in \Gamma'$ and
$$f(x) = Mx - Mx_1 - Mx_2 + y_1 + y_2.$$
Denote $\psi(x) = Mx - Mx_1 - Mx_2 + y_1 + y_2$.
Recalling that $|\Gamma'| \geq |\Gamma|/P_{BSG}^{(1)}(K)$ and
$$2^{3m}/K \leq E(\Gamma) \leq |\Gamma|^3,$$
we obtain that $f(x) = \psi(x)$ for at least
$$\frac{|\Gamma'|}{(K')^2} = \frac{|\Gamma'|}{P_1\big(P_{BSG}^{(1)}(K) P_{BSG}^{(2)}(K)\big)^2} \geq \frac{2^m}{K P_{BSG}^{(1)}(K) P_1\big(P_{BSG}^{(1)}(K) P_{BSG}^{(2)}(K)\big)^2}$$
values of $x\in S$.
The theorem follows with $P_4(K) = K P_{BSG}^{(1)}(K) P_1\big(P_{BSG}^{(1)}(K) P_{BSG}^{(2)}(K)\big)^2$.
\end{proof}

The main technical lemma for our proofs is an algorithmic version of this last result, which relies on the quadratic Goldreich-Levin theorem (see Theorem~\ref{thm:quadratic_GL}).

\begin{lemma}[Algorithmic restricted homomorphism]
\label{lem:findingaffine}
Suppose $S\subseteq \F_2^m$ and $f: S \to \F_2^n$ satisfy
$$\big|\big\{(x_1, x_2, x_3, x_4)\in S^4:\: x_1+x_2 = x_3+x_4 \,\text{ and }\, f(x_1)+f(x_2) = f(x_3)+f(x_4)\big\}\big| \geq 2^{3m}/K.$$
There is a randomized algorithm that makes $K^{O(\log K)} (m+n)^2 \log (m+n)$ queries to $S$ and to $f$, runs in $K^{O(\log K)} (m+n)^3 \log (m+n)$ time and, with probability at least $0.7$, returns $M\in \F_2^{n\times m}$, $v\in \F_2^n$ such that
$$\big|\big\{x\in S:\: f(x) = Mx+v\big\}\big| \geq 2^m/P_4'(K).$$
\end{lemma}

\begin{proof}
Define the function $g: \F_2^{m+n} \to \{-1, 0, 1\}$ by
$$g(x, y) = \one_S(x)\cdot (-1)^{f(x)\cdot y}.$$
Note that one query to $g$ can be made using one query to $S$, one query to $f$ and $O(n)$ time.
We first show that $g$ correlates well with a quadratic function:

\begin{claim}
\label{claim:highu3}
    There exists a quadratic polynomial $p: \F_2^{m+n} \to \F_2$ such that
    $$\Big|\Exp_{\substack{x\in \F_2^m,\\ y\in \F_2^n}} g(x,y) (-1)^{p(x,y)}\Big| \geq \frac{1}{P_4(K)},$$
    where $P_4(\cdot)$ is the polynomial promised by Lemma~\ref{lem:closeaffine}.
\end{claim}

\begin{proof}
From Lemma~\ref{lem:closeaffine}, we know there exists an affine-linear function $\psi: \F_2^m \to \F_2^n$ such that
$$\Pr_{x\in \F_2^m} \big[x\in S \,\text{ and }\, f(x) = \psi(x)\big] \geq \frac{1}{P_4(K)}.$$
Let $E$ be the set where $f$ and $\psi$ agree:
$$E = \big\{x\in S:\: f(x) = \psi(x)\big\}.$$
Note that $g(x, y) = (-1)^{\psi(x)\cdot y}$ for all $x\in E$, $y\in \F_2^n$, and so by Cauchy-Schwarz
\begin{align*}
    \Exp_{x\in \F_2^m} \one_E(x)
    &= \Exp_{x\in \F_2^m} \Big(\one_E(x) \cdot \Exp_{y\in \F_2^n} g(x,y) (-1)^{\psi(x)\cdot y}\Big) \\
    &\leq \Big(\Exp_{x\in \F_2^m} \one_E(x)^2 \Big)^{1/2} \Big(\Exp_{x\in \F_2^m}\Big(\Exp_{y\in \F_2^n} g(x,y) (-1)^{\psi(x)\cdot y}\Big)^2\Big)^{1/2} \\
    &= \Big(\Exp_{x\in \F_2^m} \one_E(x) \Big)^{1/2} \Big(\Exp_{x\in \F_2^m} \Exp_{y, y'\in \F_2^n} g(x,y) g(x,y') (-1)^{\psi(x)\cdot (y+y')}\Big)^{1/2} \\
    &= \Big(\Exp_{x\in \F_2^m} \one_E(x) \Big)^{1/2} \Big(\Exp_{x\in \F_2^m} \Exp_{z\in \F_2^n} \one_S(x) (-1)^{f(x) \cdot z} (-1)^{\psi(x)\cdot z}\Big)^{1/2}.
\end{align*}
We conclude that
$$\Big|\Exp_{x\in \F_2^m} \Exp_{z\in \F_2^n} g(x,z) (-1)^{\psi(x)\cdot z}\Big| \geq \Exp_{x\in \F_2^m} \one_E(x) = \Pr_{x\in \F_2^m}[x\in E] \geq \frac{1}{P_4(K)}.$$
The quadratic function $p: (x,z) \mapsto \psi(x)\cdot z$ thus satisfies the claim.
\end{proof}

We now use the quadratic Goldreich-Levin theorem (Theorem~\ref{thm:quadratic_GL}) with $f$ replaced by $g$ and $\varepsilon := 1/(2P_4(K))$.
We conclude that, in $O((m+n)^3)$ time and using $(m+n)^2 \log (m+n) \cdot K^{O(\log(K))}$ queries to $g$, we can obtain a quadratic function $q: \F_2^{m+n} \rightarrow \F_2$ which satisfies the following with probability at least $0.9$:
\begin{equation}
\label{eq:highcorrelation}
    \Big|\Exp_{x\in \F_2^m,\, y\in \F_2^n} \one_S(x) (-1)^{f(x)\cdot y} (-1)^{q(x,y)}\Big| \geq \frac{1}{2 P_4(K)}.
\end{equation}
Assume that this inequality holds, and write
$$q(x, y) = (x,y)^T A (x,y) + u\cdot x + u'\cdot y + b,$$
where $A \in \F_2^{(m+n) \times (m+n)}$, $u\in \F_2^m$, $u'\in \F_2^n$ and $b\in \F_2$.
Denote the $(m\times n)$-submatrix of $A$ defined by its first $m$ rows and last $n$ columns by $A_{12}$, and the $(n\times m)$-submatrix of $A$ defined by its last $n$ rows and first $m$ columns by $A_{21}$.
We claim that $f$ agrees often with an affine-linear function whose linear part equals $(A_{12}^T + A_{21}) x$:

\begin{claim}
    If equation~\eqref{eq:highcorrelation} holds, then there exists some $z_0\in \F_2^n$ such that
    \begin{equation}
    \label{eq:interim_agreement_funcs}
    \big|\big\{x\in S:\: f(x) = (A_{12}^T + A_{21}) x + z_0\big\}\big| \geq \frac{2^m}{64 P_4(K)^3}.
\end{equation}
\end{claim}

\begin{proof}
Define the bilinear form $B: \F_2^m \times \F_2^n \to \F_2$ by
$$B(x,y) = q(x, y) - q(x, 0) - q(0, y) + q(0, 0).$$
We have that
\begin{align*}
    B(x,y) &= \begin{bmatrix} x & y \end{bmatrix} A \begin{bmatrix} x \\ y \end{bmatrix} + \begin{bmatrix} u & u' \end{bmatrix} \begin{bmatrix} x \\ y \end{bmatrix} + b 
    - \begin{bmatrix} x & 0^n \end{bmatrix} A \begin{bmatrix} x \\ 0^n \end{bmatrix} - \begin{bmatrix} u & u' \end{bmatrix} \begin{bmatrix} x \\ 0^n \end{bmatrix} - b  \\
    &\qquad - \begin{bmatrix} 0^m & y \end{bmatrix} A \begin{bmatrix} 0^m \\ y \end{bmatrix} - \begin{bmatrix} u & u' \end{bmatrix} \begin{bmatrix} 0^m \\ y \end{bmatrix} - b  + b \\
    &= \left( \begin{bmatrix} x & 0^n \end{bmatrix} + \begin{bmatrix} 0^m & y \end{bmatrix} \right) A \begin{bmatrix} x \\ y \end{bmatrix}  - \begin{bmatrix} x & 0^n \end{bmatrix} A \begin{bmatrix} x \\ 0^n \end{bmatrix} -  \begin{bmatrix} 0^m & y \end{bmatrix} A \begin{bmatrix} 0^m \\ y \end{bmatrix} \\
    &= \begin{bmatrix} x & 0^n \end{bmatrix} A \begin{bmatrix} 0^m \\ y \end{bmatrix} + \begin{bmatrix} 0^m & y \end{bmatrix} A \begin{bmatrix} x \\ 0^n \end{bmatrix} \\
    &= x^T A_{12} y + y^T A_{21} x \\
    &= y^T (A_{12}^T + A_{21}) x.  
\end{align*}
Write $\sigma := 1/(2 P_4(K))$ and $M := A_{12}^T + A_{21}$ for convenience, so that $B(x,y) = Mx \cdot y$.
Plugging in
$$q(x,y) = Mx\cdot y + q(x,0) + q(0,y) - q(0,0)$$
into  Eq.~\eqref{eq:highcorrelation}, we obtain
$$\Big|\Exp_{x\in \F_2^m} \Exp_{y\in \F_2^n} \one_S(x) (-1)^{f(x)\cdot y} (-1)^{Mx\cdot y} (-1)^{q(x,0) + q(0,y) - q(0,0)}\Big| \geq \sigma.$$
By the triangle inequality, we conclude that
\begin{equation}
\label{eq:simplifyingafterbilinear}
\sum_{x \in S} \Big| \Exp_{y \in \mathbb{F}_2^n} (-1)^{f(x)\cdot y} (-1)^{Mx\cdot y} (-1)^{q(0,y)} \Big| \geq \sigma\cdot 2^m,
\end{equation}
where we used the fact that for a fixed $x\in S$, the quantity $(-1)^{q(x,0)-q(0,0)}$ is constant and has absolute value $1$.

Defining the function $h: \F_2^n \rightarrow \{-1,1\}$ by $h(y)=(-1)^{q(0, y)}$, one can rewrite Eq.~\eqref{eq:simplifyingafterbilinear} as
$$\sum_{x \in S} \big| \widehat{h}\big(f(x) + Mx\big) \big| \geq \sigma\cdot 2^m.$$
Since $|\widehat{h}(z)| \leq 1$ for all $z\in \F_2^n$, this implies that there exist at least $(\sigma/2)\cdot 2^m$ many $x \in S$ such that $\big| \widehat{h}\big(f(x) + Mx\big) \big| \geq \sigma/2$.
Let us define the set $T = \big\{z \in \F_2^n:\: |\widehat{h}(z)| \geq \sigma/2\big\}$, so that
$$\big|\big\{x\in S:\: f(x)+Mx \in T\big\}\big| \geq \frac{\sigma 2^m}{2}.$$
Then
$$\frac{\sigma 2^m}{2} \leq \sum_{z\in T} \big|\big\{x\in S:\: f(x)+Mx = z\big\}\big| \leq |T| \cdot \max_{z_0\in T} \big|\big\{x\in S:\: f(x)+Mx = z_0\big\}\big|.$$
Since $h$ is a Boolean function, by Parseval we have that
$$1 = \sum_{z \in \F_2^n} |\widehat{h}(z)|^2 \geq \sum_{z \in T} (\sigma/2)^2 \implies |T| \leq 4/\sigma^2.$$
From these last two inequalities, we conclude there exists some $z_0\in T$ such that
$$\big|\big\{x\in S:\: f(x)+Mx = z_0\big\}\big| \geq \frac{1}{|T|} \frac{\sigma 2^m}{2} \geq \frac{\sigma^3 2^m}{8},$$
which proves the claim.
\end{proof}

It now suffices to find such a vector $z_0\in \F_2^n$ such that  Eq.~\eqref{eq:interim_agreement_funcs} holds.
We do this by sampling $x_1, x_2, \ldots, x_t$ uniformly at random from $\FF_2^m$, passing each point $x_i$ through the indicator function $\one_S$ and then computing the difference $d(x_i) = f(x_i) - (A_{12}^T + A_{21}) x_i$. 
For each $z \in \{d(x_i)\}_{i \in [t]}$, we then estimate $\Pr_{x \in S}\big[f(x) = (A_{12}^T + A_{21})x + z\big]$ and output the value $z^*$ which maximizes the agreement.
To complete the argument, let us now comment on the value of $t$ required to determine a good value of $z^*$.
First, note that  Eq.~\eqref{eq:interim_agreement_funcs} implies
$$\Pr_{x \in \FF_2^m}[d(x) = z_0] \geq \frac{1}{64 P_4(K)^3}.$$
Thus, by sampling $t = O(P_4(K)^3)$ times, we ensure that $v_0 \in \{d(x_i)\}_{i \in [t]}$ with probability at least $0.9$.
Finally, we determine $z^*$ as mentioned before by estimating $\Pr_{x \in S}\big[f(x) = (A_{12}^T + A_{21})x + z\big]$ for each $z \in \{d(x_i)\}_{i \in [t]}$, which can be done up to error $1/(128 P_4(K)^3)$ with probability at least $1-0.1/t$ using an empirical estimator\footnote{In particular, for Boolean functions $f,g$, one can estimate $\Pr_{x}[f(x) = g(x)]$ up to error $\varepsilon$ with probability at least $1-\delta$ using the empirical estimate
$\mathrm{Est}_m := \frac{1}{m} \sum_{j =1}^m f(x_j) g(x_j)$
which can be computed by querying $f,g$ at uniformly random  $x_1,\ldots,x_m \in \FF_2^n$ and for $m = \poly(1/\varepsilon \log(1/\delta))$.}
that uses $O(\log(K) P_4(K)^3)$ samples from $\FF_2^n$ and queries to $S$ and $f$ for each $i \in [t]$.
In total, this procedure consumes $O(\log(K) P_4(K)^{6})$ queries to $S$ and to $f$, and succeeds with probability at least $0.8$ (after taking the union bound).

We then return $M = A_{12}^T + A_{21}$ and $v = z^*$ as given above.
With probability at least $0.7$, the guarantee of the statement is satisfied with $P_4'(K) = 128 P_4(K)^3$.
%(To obtain probability at least $0.9$, we can repeat the procedure twice and output the one that attains the highest empirical agreement with $f$.)
The overall query and time complexities of the algorithm are mostly due to the use of Theorem~\ref{thm:quadratic_GL}, and match the complexities stated in the lemma.
\end{proof}

\section{Proof of the algorithmic $\PFR$ theorems}\label{sec:algo_PFR}
In this section, we provide the proof of our main results, relying on the algorithmic lemmas established in Section~\ref{sec:lemmas}. We begin with classical algorithms in Section~\ref{sec:classical} and then proceed to quantum algorithms in Section~\ref{sec:quantum}.

\subsection{Classical algorithmic $\PFR$ theorems}
\label{sec:classical}
We first prove our main classical results, which we restate more precisely below.
%\classicalPFR*

\begin{theorem}[Algorithmic $\PFR$]
\label{thm:algoPFR1}
    Suppose $A \subseteq \mathbb{F}_2^n$ satisfies $|A + A| \leq K|A|$.
    There is a randomized algorithm that takes $O(\log|A| + K)$ random samples from $A$, makes $2^{O(K)} \log^2 |A|\cdot \log\log |A|$ queries to $A$, runs in time $K^{O(\log K)} n^4 \log n$ and has the following guarantee:
    with probability at least $2/3$, it outputs a basis for a subspace $V \leq \mathbb{F}_2^n$ of size $|V| \leq |A|$ such that $A$ can be covered by $P_1'(K)$ translates of $V$.
\end{theorem}

\begin{proof}
We first describe the algorithm to find $V$: 
\begin{enumerate}
    \item Sample $t = 28 \log|A| + 56K$ many uniformly random elements from $A$, and denote their linear span by $U$.
    Let $A' := A\cap U$.
    \item Take a random linear map $\pi: U \to \mathbb{F}_2^m$ where $m = \log|A| + 4\log K + 10$. Let $S = \pi(A')$ denote the image of $A'$ under $\pi$, and let $f: S \to U$ be the inverse of $\pi$ when restricted to $S$.\footnote{In our analysis we show that this inverse is well-defined with high probability.}
    \item Apply Lemma~\ref{lem:findingaffine} to obtain an affine-linear map $\psi: \mathbb{F}_2^m \to U$ such that $f(x) = \psi(x)$ for at least $|A|/P_4'\big(2^{33}K^{13}\big)$ values $x \in S$.
    \item Take a subspace $V$ of $\textsf{Im}(\psi)$ having size at most $|A|$, and output a basis for $V$.
\end{enumerate}

We now analyze the correctness and complexity of this algorithm.
For Step $(1)$, note that Theorem~\ref{thm:spanAsize} directly implies that $|\vspan(A)|\leq 2^{2K}\cdot |A|$.
Now, by our choice of $t$, Lemma~\ref{lem:spansample} implies that  $|A'| \geq |A|/2$ with probability at least $0.99$.
Supposing this is the case, we have that
$$
|A' + A'|\leq |A+A|\leq  K|A| \leq 2K |A'|.
$$
Moreover, by Lemma~\ref{lem:4Asize} we conclude that  $|4A'| \leq |4A| \leq K^4 |A| \leq 2K^4 |A'|$.

For Step $(2)$, note that Lemma~\ref{lem:randomFreiman} shows that (with probability at least $0.99$) $\pi$ is a Freiman isomorphism from $A'$ to $S = \pi(A')$.
In this case, the inverse map $f: S\to A'$ is a Freiman isomorphism and $|S| = |A'|$.\footnote{We remark that the definition of Freiman isomorphism gives these two properties immediately.}

In Step $(3)$ we wish to apply Lemma~\ref{lem:findingaffine}, which requires us to bound from below the quantity
$$\big|\big\{(x_1, x_2, x_3, x_4)\in S^4:\: x_1+x_2 = x_3+x_4 \,\text{ and }\, f(x_1)+f(x_2) = f(x_3)+f(x_4)\big\}\big|.$$
We claim that this is at least $|A'|^3/(2K)$:

\begin{claim}
    If $f: S \to A'$ is a Freiman isomorphism and $|A'+A'| \leq 2K |A'|$, then
    $$\big|\big\{(x_1, x_2, x_3, x_4)\in S^4:\: x_1+x_2 = x_3+x_4 \,\text{ and }\, f(x_1)+f(x_2) = f(x_3)+f(x_4)\big\}\big| \geq \frac{|A'|^3}{2K}.$$
\end{claim}

\begin{proof}
If $f$ is a Freiman isomorphism, then the quantity above equals
\begin{align*}
    &\big|\big\{(x_1, x_2, x_3, x_4)\in S^4:\: f(x_1)+ f(x_2) = f(x_3)+f(x_4)\big\}\big| \\
    &= \big|\big\{(y_1, y_2, y_3, y_4)\in (A')^4:\: y_1+y_2 = y_3+y_4\big\}\big| \\
    &= E(A').
\end{align*}
Now note that
$$\sum_{z\in 2A'} \big|\big\{(y_1, y_2)\in (A')^2:\: y_1+y_2 = z\big\}\big| = |A'|^2$$
and
\begin{align*}
    &\sum_{z\in 2A'} \big|\big\{(y_1, y_2)\in (A')^2:\: y_1+y_2 = z\big\}\big|^2 \\
    &= \sum_{z\in 2A'} \big|\big\{(y_1, y_2, y_3, y_4)\in (A')^4:\: y_1+y_2 = z = y_1+y_2\big\}\big|^2 \\
    &= E(A'),
\end{align*}
so by Cauchy-Schwarz
$$|A'|^2 \leq |2A'|^{1/2} E(A')^{1/2} \implies E(A') \geq |A'|^4/|2A'|.$$
The claim now follows from the assumption $|2A'| \leq 2K |A'|$.
\end{proof}

Next we note that, by assumption and by our choice for $m$, we have that
$$|A'| \geq \frac{|A|}{2} \geq \frac{2^m}{2^{11}K^4}.$$
From the claim, we then conclude that $S$ and $f$ satisfy the hypothesis of Lemma~\ref{lem:findingaffine} with $K$ substituted by $K' := 2^{33}K^{13}$.
We can then obtain an affine-linear function $\psi: \FF_2^m \to U$ such that, with probability at least $0.7$,
\begin{align}
    \label{eq:promiseofpsi=phi}
    \big|\big\{x\in S:\: f(x) = \psi(x)\big\}\big| \geq \frac{2^m}{P_4'\big(2^{33}K^{13}\big)}.
\end{align}
It remains to argue how one can simulate queries to $S$ and $f$, as required by the statement of Lemma~\ref{lem:findingaffine}.
To this end, observe that we have a full description of the linear map~$\pi: U\to \FF_2^m$, so in time $O(m^2 n)$ we can find its kernel $\ker(\pi)= \{v\in U:\: \pi(v)=0\}$.
We first make three observations about this:
$(a)$ $\ker(\pi)$ is a subspace of size
$$\frac{|U|}{|\mathrm{Im}(\pi)|} \leq \frac{|\vspan(A)|}{|S|} \leq \frac{2|\vspan(A)|}{|A|} \leq 2^{2K},$$
where we used Theorem~\ref{thm:spanAsize} in the final inequality;
$(b)$ for every $x\in \mathrm{Im}(\pi)$, we have that $\pi^{-1}(x)$ is a translate of $\ker(\pi)$;
$(c)$ in $O(m^2 n)$ time we can find the inverse map $\pi^{-1}: \mathrm{Im}(\pi) \to U/\ker(\pi)$.
Using item $(b)$, we can check whether $x\in S$ (i.e., $\pi^{-1}(x) \cap A \neq \emptyset$) by enumerating over all $y\in \pi^{-1}(x)$ and checking if $y\in A$ or not.
By item $(a)$, this takes at most $2^{2K}$ queries to $A$.
Hence, after computing $\ker(\pi)$ and $\pi^{-1}$, one can make one query to $S$ and to $f$ using $2^{2K}$ queries to $A$ and $O(mn)$ time.

Now define the affine subspace $V' = \mathrm{Im}(\psi)$.
By definition we have that  $|V'| \leq 2^m = O(K^4)|A'|$, and by  Eq.~\eqref{eq:promiseofpsi=phi} we have
$$|A\cap V'| = |\mathrm{Im}(f) \cap \mathrm{Im}(\psi)| \geq |S| \Pr_{x\in S}[f(x) = \psi(x)]\geq \frac{2^m}{P_4'\big(2^{33}K^{13}\big)}.$$
It follows that
$$|A + (A\cap V')| \leq |A+A| \leq K |A| \leq 2^m \leq P_4'\big(2^{33}K^{13}\big) |A\cap V'|.$$
Applying Ruzsa's covering lemma (Lemma~\ref{lem:ruzsacovering}), we obtain that $A$ can be covered by $P_4'\big(2^{33}K^{13}\big)$ translates of $2(A\cap V') \subseteq V' + V' =\psi(0)+V'$.

Overall, the complexity of the algorithm is as follows: we need $O(K+\log |A|)$ samples from $A$, and the number of queries to $A$ is as given by Lemma~\ref{lem:findingaffine}, i.e.,
$$2^{2K} \cdot K^{O(\log K)} (m+\log|U|)^2 \log (m+\log|U|) = 2^{O(K)} (\log |A|)^2 \log \log |A|,$$
where we used that $m = \log|A| + O(\log K)$ and $\log|U| = O(\log|A| + K)$.
The total runtime is the cost of Lemma~\ref{lem:findingaffine}, the cost of inverting $\pi$, and the cost for making the queries to $f$ and $S$, i.e.,
$$
K^{O(\log K)} (m+n)^3 \log (m+n) + O(m^2 n) + K^{O(\log K)} (m+n)^2 \log (m+n) \cdot O(mn) = K^{O(\log K)} n^4 \log n,
$$
proving the theorem statement.
\end{proof}

\begin{theorem}[Homomorphism testing]
\label{thm:algoPFR2}
    Suppose $f: \F_2^m \to \F_2^n$ satisfies
    $$\Pr_{x_1+x_2 = x_3+x_4} \big[f(x_1)+f(x_2) = f(x_3)+f(x_4)\big] \geq 1/K.$$
    There is a randomized algorithm that makes $K^{O(\log K)} (m+n)^2 \log (m+n)$ queries to $f$, runs in $K^{O(\log K)} (m+n)^3 \log (m+n)$ time and, with probability at least $2/3$, outputs a matrix $M \in \F_2^{n\times m}$ and a vector $v\in \F_2^n$ such that
    $$\Pr_{x\in \F_2^m} \big[f(x) = Mx + v\big] \geq 1/P_2'(K).$$
\end{theorem}

\begin{proof}
This follows immediately from Lemma~\ref{lem:findingaffine} with $S = \F_2^m$ and $P_2'(K) = P_4'(K)$.
\end{proof}

\begin{theorem}[Structured approximate homomorphism]
\label{thm:algoPFR3}
    Suppose $f: \F_2^m \to \F_2^n$ satisfies
    $$\big|\big\{f(x)+f(y)-f(x+y):\: x, y\in \F_2^m \big\}\big| \leq K.$$
    There is a randomized algorithm that makes $K^{O(\log K)} (m+n)^2 \log (m+n)$ queries to $f$, runs in $K^{O(\log K)} (m+n)^3 \log (m+n)$ time and, with probability at least $2/3$, outputs a matrix $M \in \F_2^{n\times m}$ such that
    $$|\{f(x) - Mx:\: x\in \F_2^m\}| \leq P_3'(K).$$
\end{theorem}

\begin{proof}
We first show that the property in the statement implies that
\begin{equation} \label{eq:propPFR2}
    \Pr_{x_1+x_2 = x_3+x_4} \big[f(x_1)+f(x_2) = f(x_3)+f(x_4)\big] \geq \frac{1}{K}.
\end{equation}
Indeed, denote $\Delta f := \big\{f(x)+f(y)-f(x+y):\: x, y\in \F_2^m \big\}$, so that $|\Delta f| \leq K$ by assumption.
Then
\begin{align*}
    \Exp_{b\in \Delta f} \Exp_{x\in \F_2^m} \Exp_{y\in \F_2^m} \one \big[f(x)+f(y)-f(x+y) = b\big] &= \frac{1}{|\Delta f|} \Exp_{x, y\in \F_2^m} \sum_{b\in \Delta f} \one \big[f(x)+f(y)-f(x+y) = b\big] \\
    &= \frac{1}{|\Delta f|} \Exp_{x, y\in \F_2^m} 1 \\
    &\geq \frac{1}{K},
\end{align*}
and so by Cauchy-Schwarz
\begin{align*}
    \frac{1}{K^2}
    &\leq \Exp_{b\in \Delta f} \Exp_{x\in \F_2^m} \Big(\Exp_{y\in \F_2^m} \one \big[f(x)+f(y)-f(x+y) = b\big]\Big)^2 \\
    &= \Exp_{b\in \Delta f} \Exp_{x\in \F_2^m} \Exp_{y, z\in \F_2^m} \one \big[f(x)+f(y)-f(x+y) = b = f(x)+f(z)-f(x+z)\big] \\
    &= \frac{1}{K} \Exp_{x, y, z\in \F_2^m} \sum_{b\in \Delta f} \one \big[f(y)-f(x+y) = b-f(x) = f(z)-f(x+z)\big] \\
    &= \frac{1}{K} \Exp_{x, y, z\in \F_2^m} \one \big[f(y)-f(x+y) = f(z)-f(x+z)\big] \\
    &= \frac{1}{K} \Exp_{x_1+x_2 = x_3+x_4} \one \big[f(x_1)+f(x_2) = f(x_3)+f(x_4)\big],
\end{align*}
which gives inequality~\eqref{eq:propPFR2} as desired.

We may then apply Lemma~\ref{lem:findingaffine} (with $S = \F_2^m$) to obtain a matrix $M \in \F_2^{n\times m}$ and a vector $v\in \F_2^n$ such that, with probability at least $0.7$, we have
\begin{equation}
\label{eq:conclPFR2}
    \Pr_{x\in \F_2^m} \big[f(x) = Mx + v\big] \geq 1/P_4'(K).
\end{equation}
We claim that, if this inequality holds (and $|\Delta f| \leq K$), then
\begin{equation}
\label{eq:conclPFR3}
    |\{f(x) - Mx:\: x\in \F_2^m\}| \leq K^2 P_4'(K),
\end{equation}
which is the property we want with $P_3'(K) = K^2 P_4'(K)$.
It then suffices to prove~\eqref{eq:conclPFR3}.

Denote $E := \big\{x\in \F_2^m:\: f(x) = Mx + v\big\}$, so that $|E| \geq 2^m/P_4'(K)$ by  Eq.~\eqref{eq:conclPFR2}.
Then
$$|\F_2^m + E| = 2^m \leq P_4'(K)\cdot |E|,$$
so we may use Ruzsa's covering lemma (Lemma~\ref{lem:ruzsacovering}, with $S = E$ and $T = \F_2^m$) to conclude there exists a set $X\subseteq \F_2^m$ of size $P_4'(K)$ such that $\F_2^m \subseteq X + 2E$.
In other words, every element of $\F_2^m$ can be written as $x+y+z$ with $x\in X$ (where $|X| \leq P_4'(K)$) and $y, z\in E$.

Now, for every $x\in X$, $y, z\in E$, by definition of the set $\Delta f$ there exist $b, b'\in \Delta f$ such that
$$f(x+y)-f(x)-f(y) = b \quad \text{and} \quad f(x+y+z)-f(x+y)-f(z) = b'.$$
Summing these two identities, we conclude that
\begin{align*}
    f(x+y+z)
    &= f(x) + f(y) + f(z) + b+b' \\
    &= f(x) + My + Mz + b+b' \\
    &= f(x) + M(x+y+z) - Mx + b+b',
\end{align*}
and so
$$f(x+y+z) - M(x+y+z) = f(x) - Mx + b+b' \in \big\{f(x') - Mx':\: x'\in X\big\} + \Delta f+\Delta f$$
is contained in a set of size at most $|X|\cdot |\Delta f|^2 \leq K^2 P_4'(K)$.
This gives  Eq.~\eqref{eq:conclPFR3} and concludes the proof of the theorem.
\end{proof}

\subsection{Quantum algorithmic $\PFR$ theorem}
\label{sec:quantum}
We now give a quantum algorithm that is quadratically better than in the query complexity compared to the classical algorithm shown in the section above. We restate the statement of the quantum result in more detail below.

\begin{theorem}[Quantum algorithmic $\PFR$]
\label{thm:algoquantumPFR1}
    Suppose $A \subseteq \mathbb{F}_2^n$ satisfies $|A + A| \leq K|A|$.
    There is a quantum algorithm that takes $O(\log|A| + K)$ random samples from $A$, makes $2^{O(K)} \log |A|$ quantum queries to $A$, runs in time $K^{O(\log K)} n^3$ and has the following guarantee:
    with probability at least $2/3$, it outputs a basis for a subspace $V \leq \mathbb{F}_2^n$ of size $|V| \leq |A|$ such that $A$ can be covered by $P_1'(K)$ translates of $V$.
\end{theorem}

To prove the above theorem, we will reprove Lemma~\ref{lem:findingaffine} in the quantum setting, but now taking advantage of the main result (Theorem~\ref{thm:sitanbootstrapping}) of Chen, Gong, Ye, and Zhang~\cite{chen2024stabilizer}, which allows us to find the closest stabilizer state to a given unknown $n$-qubit quantum state. Formally, the quantum version of Lemma~\ref{lem:findingaffine} is as follows.

\begin{lemma}
\label{lem:findingfreimanquantum}
Suppose $S\subseteq \F_2^m$ and $f: S \to \F_2^n$ satisfy
$$\big|\big\{(x_1, x_2, x_3, x_4)\in S^4:\: x_1+x_2 = x_3+x_4 \,\text{ and }\, f(x_1)+f(x_2) = f(x_3)+f(x_4)\big\}\big| \geq 2^{3m}/K.$$
There is a quantum algorithm that makes $K^{O(\log K)} (m+n)$ quantum queries to $S$ and to $f$, runs in $K^{O(\log K)}(m+n)^3$ time and, with probability at least $0.7$, returns $M\in \F_2^{n\times m}$, $v\in \F_2^n$ such that
$$\big|\big\{x\in S:\: f(x) = Mx+v\big\}\big| \geq 2^m/P_4'(K).$$%\anote{Tom, Davi: Just calling your attention to the polynomial $P_4'$ which I will let you decide on how to handle, depending on the choice you make regarding the convention of the polynomials. Just note that in the classical case, we had Eq.6 going into Claim 3.6 which was a promise of $1/(2P_4(K))$. Here, we have Eq. 25 which has a promise of $1/(2P_4(K))^3$ before we would use Claim 3.6.}
\end{lemma}

To prove Lemma~\ref{lem:findingfreimanquantum} and describe its corresponding algorithm, we need a quantum protocol to prepare the quantum state that encodes the function 
$$
g_S(x,y) = \one_S(x) (-1)^{f(x)\cdot y},
$$
which from Claim~\ref{claim:highu3}, we know has high correlation with a quadratic function.

\begin{claim}\label{claim:prep_psi}
Consider the context of Lemma~\ref{lem:findingfreimanquantum}. Let $\delta \in (0,1)$. Suppose we have quantum query access to $S$ via the oracle $O_S$ and query access to  $f: S \rightarrow \F_2^n$ via the oracle $O_f$ as~follows
$$
\ket{x,0} \stackrel{O_S}\longrightarrow \ket{x,\one_S(x)}, \quad \ket{x,0^n} \stackrel{O_f}\longrightarrow \ket{x,f(x)}.
$$
Then, there exists a quantum algorithm that prepares the $(m+n)$-qubit state $\ket{\psi}$ encoding $g_S(x,y)$ for $(x,y) \in \FF_2^m \times \F_2^n$ as follows
$$
\ket{\psi}= \frac{1}{\sqrt{2^n |S|}}\sum_{\substack{x\in S,\\y\in \F_2^n}}(-1)^{f(x)\cdot y}\ket{x,y}.
$$
The algorithm makes $O(K \log(1/\delta))$ queries to $O_S, O_f$ and uses $O(K n \log(1/\delta))$ gate complexity to prepare one copy of $\ket{\psi}$ with probability at least $1-\delta$.
\end{claim}
\begin{proof}
First, given quantum query access to $S$, the algorithm prepares
$$
\frac{1}{\sqrt{2^m}}\sum_{x\in \F_2^m} \ket{x,0} \stackrel{O_S}{\longrightarrow} \frac{1}{\sqrt{2^m}}\sum_{x\in \F_2^m} \ket{x,\one_S(x)},
$$
and measures the second register. With probability $|S|/2^m \geq 1/K$, the algorithm obtains $1$ in which case the resulting state is $\ket{S}=\frac{1}{\sqrt{|S|}}\sum_{x\in S}\ket{x}$. So making $O(K \log(1/\delta))$ quantum queries, one can prepare the state $\ket{S}$ with probability at least $1-\delta$. After this, the algorithm simply performs the following
\begin{align}
\frac{1}{\sqrt{|S|}}\sum_{x\in S}\ket{x}\otimes \frac{1}{\sqrt{2^n}}\sum_{y\in \F_2^n}\ket{y} &\stackrel{O_f}\longrightarrow \frac{1}{\sqrt{2^n|S|}}\sum_{\substack{x\in S,\\y\in \F_2^n}}\ket{x,y,f(x)}\\
&\longrightarrow \frac{1}{\sqrt{2^n|S|}}\sum_{\substack{x\in S,\\y\in \F_2^n}}\ket{x,y,f(x)}\otimes_{i=1}^n\ket{f(x)_i\cdot y_i}\\
&\longrightarrow \frac{1}{\sqrt{2^n|S|}}\sum_{\substack{x\in S,\\y\in \F_2^n}}\ket{x,y}\ket{f(x)}\otimes_{i=1}^n\ket{f(x)_i\cdot y_i} \ket{f(x)\cdot y}.
\label{eq:interim_state}
\end{align}
where the second operation is by applying  $n$ many CCNOT gates with the control qubits being $y_i,f(x)_i$ applied onto the target qubit $\ket{0}_i$, and the third operation is by applying $n$ CNOT gate between the control qubit $\ket{f(x)_i\cdot y_i}$ and target qubit $\ket{0}$.

After obtaining the state in Eq.~\eqref{eq:interim_state}, the algorithm applies a single-qubit Hadamard on the last qubit and measures in the computational basis. If it is $1$, it continues. First note that if the last qubit was $1$, then the resulting quantum state is
\begin{align}
     \frac{1}{\sqrt{2^n |S|}}\sum_{\substack{x\in S,\\y\in \F_2^n}}(-1)^{f(x)\cdot y}\ket{x,y}\ket{f(x)}\otimes_{i=1}^n\ket{f(x)_i\cdot y_i} \ket{1}.
\end{align}
Furthermore the probability of obtaining $1$ is exactly $1/2$. Upon succeeding, the algorithm inverts the $n$ many CCNOT gates and the query operator $O_f$ to obtain the state
$$
\ket{\psi}= \frac{1}{\sqrt{2^n |S|}}\sum_{\substack{x\in S,\\y\in \F_2^n}}(-1)^{f(x)\cdot y}\ket{x,y}.
$$
Overall, the algorithm used $O(Kn \log(1/\delta))$ many quantum gates and $O(K \log(1/\delta))$ queries to prepare one copy of $\ket{\psi}$.
\end{proof}

We are now ready to prove Lemma~\ref{lem:findingfreimanquantum}.
\begin{proof}[Proof of Lemma~\ref{lem:findingfreimanquantum}]
The proof will be similar to the classical proof in Lemma~\ref{lem:findingaffine}.
Similar to the proof there, we are guaranteed by Claim~\ref{claim:highu3} that there exists a \emph{quadratic} polynomial $q: \F_2^{n} \times \F_2^n \rightarrow \{0,1\}$ which has high correlation with $g_S(x,y)=\one_S(x) (-1)^{f(x)\cdot y},$ i.e.,
$$
\left| \Exp_{x,y \in \mathbb{F}_2^m \times \mathbb{F}_2^n} [\one_S(x)(-1)^{f(x)\cdot y} (-1)^{q(x, y)}] \right| \geq \frac{1}{P_4(K)},
$$
where $P_4(\cdot)$ is the polynomial promised by Lemma~\ref{lem:closeaffine}. For simplicity in notation, let us denote $\sigma:=1/P_4(K)$. In particular, defining the quantum states 
$$
\ket{\psi} = \frac{1}{\sqrt{2^n |S|}}\sum_{x\in S,\\y\in \F_2^n}(-1)^{f(x)\cdot y}\ket{x,y}, \quad \ket{\phi_q}= \frac{1}{\sqrt{2^{n+m}}} \sum_{x\in \F_2^n,y\in \F_2^m}(-1)^{q(x, y)}\ket{x,y},
$$
we have that $|\langle \psi|\phi_q\rangle|^2 \geq \sigma^{2}$. Moreover by Theorem~\ref{thm:neststabilizer}, we note that the quantum state $\ket{\phi_q}$ is a \emph{stabilizer state}\footnote{We remark that $\ket{\phi_q}$ is in fact a degree-$2$ phase state (i.e., the subspace is $\F_2^{m+n}$ and there are no complex phases), but we will not use that here.}, and thus the stabilizer fidelity of $\ket{\psi}$ is also at least $\sigma^{2}$. At this point, we use Theorem~\ref{thm:sitanbootstrapping} on copies of $\ket{\psi}$ prepared using Claim~\ref{claim:prep_psi}, with the error instantiated as $\varepsilon=\sigma^{2}/2$ to learn a stabilizer state $\ket{s}$ such that $|\langle s| \psi \rangle|^2 \geq \sigma^{2}/2$. By Theorem~\ref{thm:neststabilizer}, we can write this stabilizer state as
\begin{equation}\label{eq:stab_state_from_bootstrapping}
\ket{s} = \frac{1}{\sqrt{|A_s|}} \sum_{z \in A_s} i^{\ell_s(z)} (-1)^{q_s(z)} \ket{z},    
\end{equation}
where $A_s \subseteq \FF_2^{n+m}$ is an affine subspace, $\ell_s$ is a linear polynomial and $q_s$ is a quadratic form. To lower bound the size of $|A_s|$, we will lower bound the size of $|A_s \cap T|$ (denoting $T:= S \times \FF_2^n$):
\begin{align*}
\frac{\sigma}{\sqrt{2}} \leq |\langle \psi | s \rangle| 
&= \Big|\frac{1}{\sqrt{|A_s|\cdot|S|\cdot 2^n}} \sum_{\substack{x \in S \\ y \in \FF_2^n \\ (x,y) \in A_s}} i^{\ell(x,y)} (-1)^{q(x,y) + f(x)\cdot y} \Big| \\
&\leq \frac{1}{\sqrt{|A_s| \cdot |S|\cdot 2^n}} \sum_{(x,y) \in A_s \cap T} \Big| i^{\ell(x,y)} (-1)^{q(x,y) + f(x)\cdot y} \Big|\\
& \leq \frac{\sqrt{|A_s \cap T|}}{\sqrt{|S|\cdot 2^n}},    
\end{align*}
where we have used the triangle inequality in the second line and noted that each internal term is at most $1$ in the final inequality along with using $|A_s| \geq |A_s \cap T|$. The above result implies that $|A_s|$ is large i.e.,
\begin{equation}\label{eq:lb_size_As}
|A_s| \geq |A_s \cap T| \geq (\sigma^{3}/2) 2^{m+n},
\end{equation}
as $|S| \geq \sigma \cdot 2^m$. Writing $A_s = a + H_s$ where $H_s$ is a linear subspace, we then have $\textsf{codim}(H_s) \leq \log(2/\sigma^{3})$. To obtain a quadratic phase state $\ket{\phi_p}$ corresponding to a quadratic phase polynomial $p:\FF_2^{n+m} \rightarrow \FF_2$ that has high fidelity with $\ket{\psi}$ from the description of $\ket{s}$, we have the following observations (similar to that in \cite[Proof of Theorem~1.1]{briet2025near}) and which will inform our approach. 

Let us denote the orthogonal complement of $H_s$ as $H_s^\perp = \{x \in \FF_2^{n+m} : x \cdot h = 0, \,\,\forall h \in H_s\}$. The Fourier decomposition of $\one_{A_s}(x)$ is given by
\begin{equation}\label{eq:fourier_indicator_coset}
\one_{A_s}(x) = \frac{|H_s|}{|2^{n+m}|}\sum_{\lambda \in H_s^\perp} (-1)^{\lambda \cdot (a+x)},    
\end{equation}
which follows from the observation that
\begin{align*}
\Exp_x[\one_{A_s}(x) (-1)^{\lambda \cdot x}] = 2^{-(n+m)} \sum_{x \in H_s} (-1)^{\lambda \cdot (a+x)} 
&= |H_s| 2^{-(n+m)} (-1)^{\lambda \cdot a} \Exp_{x \in H_s}[(-1)^{\lambda \cdot x}] \\
&= |H_s| 2^{-(n+m)} (-1)^{\lambda \cdot a} [\lambda \in H_s^\perp].    
\end{align*}
We then observe that
\begin{align}
    \sigma/\sqrt{2} \leq |\la \psi | s \ra |
    &= \frac{1}{\sqrt{|A_s|\cdot|S|\cdot 2^n}} \Big| \sum_{z \in \FF_2^{n+m}} \one_{A_s}(z) g_S(z) (-1)^{q_s(z)} i^{\ell_s(z)} \Big| \\
    &= \frac{|H_s|}{\sqrt{|A_s|\cdot|S|\cdot 2^n}} \Big| \sum_{\lambda \in H_s^\perp} \Exp_{z \in \FF_2^{n+m}} (-1)^{\lambda \cdot (a+z)} g_S(z) (-1)^{q_s(z)} i^{\ell_s(z)} \Big| \\
    &\leq \frac{|H_s|\cdot|H_s^\perp|}{\sqrt{|A_s|\cdot|S|\cdot 2^n}} \max_{\lambda \in H_s^\perp} \Big|  \Exp_{z \in \FF_2^{n+m}} (-1)^{\lambda \cdot z} g_S(z) (-1)^{q_s(z)} i^{\ell_s(z)} \Big| \\
    &\leq \frac{\sqrt{2}}{\sigma^{2}} \max_{\lambda \in H_s^\perp} \Big|  \Exp_{z \in \FF_2^{n+m}} (-1)^{\lambda \cdot z} g_S(z) (-1)^{q_s(z)} i^{\ell_s(z)} \Big| \label{eq:max_lambda},
\end{align}
where we used the Fourier decomposition of $\one_{A_s}(z)$ from Eq.~\eqref{eq:fourier_indicator_coset} in the second line, applied the triangle inequality along with considering the $\lambda \in H_s^\perp$ which maximizes the expectation in the third line, and finally used Eq.~\eqref{eq:lb_size_As} as well as noting $|H_s|\cdot|H_s^\perp| = 2^{n+m}$ and $|S|\geq \sigma \cdot 2^m$. From Eq.~\eqref{eq:max_lambda}, we have that $\exists \lambda^\star \in H_s^\perp$ such that
$$
\Big|  \Exp_{z \in \FF_2^{n+m}} g_S(z) (-1)^{q_s(z) + \lambda^\star \cdot z} i^{\ell_s(z)} \Big| \geq \sigma^{3}/2.
$$
Let us define the function $h(z) := g_S(z) (-1)^{q_s(z) + \lambda^\star \cdot z}$. Additionally, we denote $R_h = \mathrm{Re}(\Exp_z[h(z)])$ to be the real part of the above expectation and $I_h = \mathrm{Im}(\Exp_z[h(z)])$ to be the imaginary part of the expectation. Now, we consider the two candidate quadratic polynomials $p_0(z) := q_s(z) + \lambda^\star \cdot z$ and $p_1(z) := q_s(z) + \lambda^\star \cdot z + \ell_s(z)$, where $q_s$ and $\ell_s$ are the quadratic and linear polynomials corresponding to the stabilizer state $\ket{s}$ in hand~(Eq.~\eqref{eq:stab_state_from_bootstrapping}). We observe that the quadratic phase states $\ket{\phi_{p_0}}$ and $\ket{\phi_{p_1}}$ satisfy
\begin{align*}
    |\la \psi | \phi_{p_0}\ra| &= \Big|  \Exp_{z \in \FF_2^{n+m}}[h(z)] \Big| = \Big|  \Exp_{z \in \FF_2^{n+m}}[h(z) \one_{\ell_s(z)=0}] + \Exp_{z \in \FF_2^{n+m}}[h(z) \one_{\ell_s(z)=1}] \Big| = |R_h + I_h| \\
    |\la \psi | \phi_{p_1}\ra| &= \Big|  \Exp_{z \in \FF_2^{n+m}}[h(z)(-1)^{\ell_s(z)}] \Big| = \Big|  \Exp_{z \in \FF_2^{n+m}}[h(z) \one_{\ell_s(z)=0}] - \Exp_{z \in \FF_2^{n+m}}[h(z) \one_{\ell_s(z)=1}] \Big| = |R_h - I_h|
\end{align*}
Noting that $\max\{|a+b|,|a-b|\} = |a| + |b| \geq \sqrt{|a|^2 + |b|^2}$, we then have
\begin{equation}\label{eq:promise_quadratics}
    \max\{|\la \psi | \phi_{p_0} \ra|, |\la \psi | \phi_{p_1} \ra|\} \geq \Big|  \Exp_{z \in \FF_2^{n+m}} g_S(z) (-1)^{q_s(z) + \lambda^\star \cdot z} i^{\ell_s(z)} \Big| \geq \sigma^{3}/2.
\end{equation}
In other words, one of the quadratic polynomials $p_0$ or $p_1$ has high correlation with $g_S(z)$. 

To determine this quadratic polynomial, we now use the following approach. We create the list of candidate quadratic polynomials $L$ where we add the polynomials $p_0^\lambda(z) := q_s(z) + \lambda \cdot z$ and $p_1^\lambda(z) := q_s(z) + \lambda \cdot z + \ell_s(z)$ for each $\lambda \in H_s^\perp$. This list will be of size $|L| = 2 |H_s^\perp| \leq 4/\sigma^{3}$, where we have used $\textsf{codim}(H_s) \leq \log(2/\sigma^{3})$. For each $p \in L$, we prepare copies of the quadratic phase state $\ket{\phi_p}$ (which is also a stabilizer state) using Lemma~\ref{lem:clifford_synthesis} and then measure $|\la \psi | \phi_p \ra|^2$ using the \textsf{SWAP} test (Lemma~\ref{lem:swap_test}) up to error $\sigma^{3}/4$ and output the quadratic polynomial $p^\star$ that maximizes the fidelity. This consumes $\poly(1/\sigma)$ sample complexity and $O(n^2/\log n \cdot\poly(1/\sigma))$ time complexity. We are guaranteed by Eq.~\eqref{eq:promise_quadratics} that $(-1)^{p^\star}$ satisfies
\begin{equation}
\Big| \Exp_{(x,y) \in \FF_2^{m} \times \FF_2^n} \one_S(x) (-1)^{p^\star(x,y) + f(x)\cdot y} \Big| \geq \frac{\sigma^{3}}{2} = \frac{1}{2P_4(K)^3},    
\end{equation}
where we have substituted back $\sigma=1/P_4(K)$ set earlier. Having determined the polynomial $p^\star$, we now proceed as in Lemma~\ref{lem:findingaffine} to determine the affine linear function $\varphi$ that agrees with $\phi$ on many values $x \in S$. This completes the proof of the lemma. The main contribution to query complexity and time complexity is utilizing Theorem~\ref{thm:sitanbootstrapping}.
\end{proof}
With this lemma, we are finally ready to prove the main theorem in this section.

\begin{proof}[Proof of Theorem~\ref{thm:algoquantumPFR1}]
The proof of this theorem is very similar to the classical proof. We write the details~below:
\begin{enumerate}
    \item Sample $t = 28 \log|A| + 56K$ many uniformly random elements from $A$, and denote their linear span by $U$.
    Let $A' := A\cap U$.
    \item Take a random linear map $\pi: U \to \mathbb{F}_2^m$ where $m = \log|A| + 4\log K + 10$. Let $S = \pi(A')$ denote the image of $A'$ under $\pi$, and let $f: S \to U$ be the inverse of $\pi$ when restricted to $S$.
    \item Apply Lemma~\ref{lem:findingfreimanquantum} to obtain an affine-linear map $\psi: \mathbb{F}_2^m \to U$ such that $f(x) = \psi(x)$ for at least $|A|/P_4'\big(2^{33}K^{13}\big)$ values $x \in S$.
    \item Take a subspace $V$ of $\textsf{Im}(\psi)$ having size at most $|A|$, and output a basis for $V$.
\end{enumerate}
% \begin{enumerate}
%     \item Sample $t=O(K+\log|A|)$ many $z_i$s from $A$.  Let $U=\vspan(\{z_1,\ldots,z_t\})$ and  $A' := A\cap U$.
%     \item Take a random linear map $\pi: U \to \mathbb{F}_2^m$ where $m = \log|A| + 4\log K + 100$. Let $S := \pi(A')$ denote the image of $A'$ under $\pi$.
%     Let $\phi: S \to U$ denote the inverse of $\pi$ when restricted to $S$.
%     \item Use  Lemma~\ref{lem:findingfreimanquantum} to find an affine-linear map $\psi: \mathbb{F}_2^m \to U$ such that $\psi(x) = \phi(x)$ for $K^{-C}|A|$ values $x \in S$.
%     \item Output a basis for the subspace $V = \psi(0) + \textsf{Im}(\psi)$.
% \end{enumerate}
The only difference between the classical and quantum algorithm is in Step $(3)$.  So, we do not reproduce the correctness analysis and refer the reader to the classical proof of Theorem~\ref{thm:classicalPFR}.

Overall, the complexity of the algorithm is as follows.
The sample complexity to the set $A$ is $O(K+\log |A|)$, as given in step $(1)$.
Computing $\ker(\pi) \leq U$ and $\pi^{-1}: \mathrm{Im}(\pi) \to U/\ker(\pi)$ takes $O(m^2n)$ time and, after this is done, each query to $S$ and $f$ takes $2^{2K}$ queries to $A$ and $O(mn)$ time.
The total number of queries to $A$ needed to apply Lemma~\ref{lem:findingfreimanquantum} is then
$$2^{2K} \cdot K^{O(\log K)} (m+\log|U|) = 2^{O(K)} \log |A|,$$
where we used that $m = \log |A| + O(\log K)$ and $\log|U| = O(\log|A| + K)$.
The total runtime is the cost of Lemma~\ref{lem:findingfreimanquantum}, the cost of inverting $\pi$ and the cost of making queries to $S$ and $f$, i.e.
$$K^{O(\log K)}(m+\log|U|)^3 + O(m^2 n) + K^{O(\log K)} (m+\log|U|) \cdot O(mn) = K^{O(\log K)} n^3,$$
concluding the proof of the theorem.
\end{proof}

\section{Lower bounds}
\label{sec:lowerbounds}
In this section, we show that the query complexities of the classical and quantum algorithms presented in Section~\ref{sec:algo_PFR} are essentially optimal in terms of their dependence on $n$. Our lower bounds are a simple information-theoretic argument (similar to the argument in~\cite[Proposition~1]{montanaro2012quantum}). We remark that both our lower bounds will consider the \emph{hard instances} to be \emph{dense} sets $A\subseteq \F_2^n$ (i.e., $|A|/2^n\geq \Omega(1)$). In this case the classical and quantum algorithm need not use random samples from $A$ at all: observe that randomly sampled points will lie in $A$ (which can be verified by queries) with constant probability, so with a constant overhead in complexity one can remove the need for samples in the classical algorithms.

\subsection{Classical lower bound}
We first show that the classical algorithm is essentially optimal in terms of the dependence in $n$.

\begin{theorem}
\label{thm:classicallowerbound}
     Suppose $A \subseteq \mathbb{F}_2^n$ satisfies $|A + A| \leq K|A|$.
     Suppose an algorithm makes $t$ queries to $O_A$ and, with non-negligible probability, identifies a subspace $H \leq \mathbb{F}_2^n$ that covers $A$ by $(2K)^C$ cosets for some constant $C>0$.
     Then $t=\Omega_K(n^2)$. 
\end{theorem}

\begin{proof}
The proof follows from letting $A$ itself be a subspace of $\F_2^n$ (clearly subspaces have doubling constant $1$). Let $\textbf{H}$ be a uniformly random subspace in $\F_2^n$. 
Given query access to $H\leq \F_2^n$, the algorithm needs to output a subspace $H'$ such that $(2K)^C$ translates of it cover $H$.
In particular, this implies that one learns a subspace $H\cap H'$ whose size is at least $2^n/(2K)^C$ (i.e., has dimension at least $n-C\log(2K)$).

Now, observe that every classical query provides one bit of information, i.e., on input $x\in \F_2^n$, outputs if $x\in H$ or not. Let the random variable $\textbf{H'}$ be the random variable corresponding to the output of the algorithm  after making $t$ queries to $O_H$. Then we have~that
$$
I(\textbf{H}:\textbf{H'})\leq t,
$$
since the algorithm makes $t$ queries each giving one bit of information. Furthermore since the goal of the algorithm is to obtain $\textbf{H}\cap \textbf{H'}$ which is a subspace of dimension $n-C\log (2K)$, by Fano's inequality~\cite{cover1999elements} we have that 
$$
\Pr[\text{identification}]\geq 1-\frac{I(\textbf{H}:\textbf{H'})}{\log |\textsf{supp}(\textbf{H})|}\geq 1-\frac{t}{\log (2^{\Theta((n-C\log (2K))^2)})} = 1 - O\Big(\frac{t}{(n-C\log (2K))^2}\Big),
$$
where we used that the number of subspaces in $\F_2^n$ is $2^{\Theta(n^2)}$. The above implies that $t=\Omega_K(n^2)$ for an algorithm that identifies $\textbf{H}$ with non-negligible probability.
\end{proof}

\subsection{Quantum lower bound}  We now show that the quantum algorithm is optimal in terms of the dependence in $n$. 
\begin{theorem}
  Suppose $A \subseteq \mathbb{F}_2^n$ satisfies $|A + A| \leq K|A|$.  Suppose an algorithm makes $t$ quantum queries to $O_A$ and, with non-negligible probability, identifies a subspace $H \leq \mathbb{F}_2^n$ that covers $A$ by $(2K)^C$ cosets for some constant $C>0$.
  Then $t=\Omega_K(n)$. 
\end{theorem}
\begin{proof} 
Like in the classical proof, we let $A$ itself be a subspace of $\F_2^n$. Let $\textbf{H}$ be a uniformly random subspace in $\F_2^n$. Recall from the previous proof that one needs to learn a subspace $H\cap H'$ whose size is at least $2^n/(2K)^C$ (i.e., has dimension at least $n - C\log(2K)$).  More formally, let $\textbf{H}$ be a uniformly random subspace in $\F_2^n$. Given quantum query access to $H\subseteq \F_2^n$, suppose the algorithm outputs $\textbf{H'}$. A quantum query algorithm can be viewed as a communication protocol, that on input $\ket{x,0}$ outputs $\ket{x, \one_H(x)}$, which uses $n+1$ qubits of communication in total.  Let the random variable $\textbf{H'}$ be the random variable corresponding to the output of the algorithm  after making $t$ queries to $O_H$. Then by Holevo's theorem, we have~that
$$
I(\textbf{H}:\textbf{H'})\leq 2t(n+1),
$$
since the communication in involves sending $n+1$ qubits back and forth, in total $t$ many times. Furthermore since the goal of the algorithm is to \emph{identify} $\textbf{H}$, by Fano's inequality~\cite{cover1999elements} we have~that 
$$
\Pr[\text{identification}]\geq 1-\frac{I(\textbf{H}:\textbf{H'})}{\log |\textsf{supp}(\textbf{H})|}\geq 1-\frac{2tn}{\log (2^{\Theta(n^2)})}=1-O(t/n),
$$
where we used that the number of subspaces in $\F_2^n$ is $2^{\Theta(n^2)}$. The above implies that $t=\Omega(n)$ for an algorithm that identifies $\textbf{H}$ with non-negligible probability. 
\end{proof}

\nocite*
\bibliographystyle{alpha}
\bibliography{refs}

\end{document}